\documentclass[reqno,11pt]{amsart}
\usepackage{amsmath,amsfonts,amssymb,amsxtra,latexsym,amscd,enumerate,amsthm,verbatim}

\usepackage[margin=1.40in]{geometry}
\numberwithin{equation}{section}

\def\ep{\epsilon}

\def\R{\mathbb{R}}

\def\C{\mathbb{C}}
\def\Z{\mathbb{Z}}

\def\N{\mathbb{N}}

\def\Z{\mathbb{Z}}

\def\ep{\epsilon}

\DeclareMathOperator{\supp}{supp}

\newtheorem{theorem}{Theorem}[section]
\newtheorem{lemma}[theorem]{Lemma}

\newtheorem{proposition}[theorem]{Proposition}
\newtheorem{definition}[theorem]{Definition}
\newtheorem{remark}[theorem]{Remark}

\def\beq{\begin{equation}}
\def\eeq{\end{equation}}

\def\beq{\begin{equation}}
\def\eeq{\end{equation}}

\begin{document}

\title[The periodic Euler--Poisson system for electrons in 2D]{Long-term regularity of the periodic Euler--Poisson system for electrons in 2D}

\author{Fan Zheng}
\address{Princeton University}
\email{fanzheng@math.princeton.edu}


\begin{abstract}
We study a basic plasma physics model--the one-fluid Euler-Poisson system on the square torus, in which a compressible electron fluid flows under its own electrostatic field. In this paper we prove long-term regularity of periodic solutions of this system in 2 spatial dimensions with small data.

Our main conclusion is that on a square torus of side length $R$,
if the initial data is sufficiently close to a constant solution,
then the solution is wellposed for a time at least $R/(\epsilon^2(\log R)^{O(1)})$, where $\epsilon$ is the size of the initial data.
\end{abstract}

\maketitle

\setcounter{tocdepth}{1}

\tableofcontents

\date{\today}

\maketitle

\section{Introduction}\label{Intro}

\subsection{Derivation of the equation}
A plasma is a collection of charged particles interacting with each other via the Coulomb forces. The plasma is the most ubiquitous form of matter in the universe, from heavenly bodies such as interstellar hydrogen and the interior of stars, to terrestrial objects like fluorescent tubes and neon signs. In addition, recent advances in controlled nuclear fusion requires a better understanding of the behavior of a plasma confined to a bounded region, for example, a tokamak fusion reactor, which resembles a torus in shape. We refer the interested reader to \cite{Bit} and \cite{DelBer} for physics references in book form.

The Euler--Poisson system describes the motion of a nonrelativistic warm adiabatic plasma consisting of electrons and ions. We assume the following in the derivation of the system.

\begin{itemize}
\item The plasma is {\it nonrelativistic}, so its dynamics follows Newton's laws.
Also, the main interaction between the ions is via the electrostatic field,
which obeys the {\bf Poisson equation}; the magnetic interaction is much smaller and can be neglected.

\item The plasma is {\it warm}, so we need to consider the thermal pressure of the electrons, arising from the temperature of the plasma.
The relation among the density, temperature, and pressure of the plasma satisfies the constitutive equation.

\item The plasma is {\it adiabatic}. This means no heat flow within the plasma.
The temperature of the plasma, though, will still vary with time. This is because the plasma is compressible, and when it is compressed, the mechanical work done on it is converted to thermal energy, so the plasma will heat up, and vice versa.

\item The plasma consists of free electrons and ions. As ions are much heavier than electrons (for example, $m_p=1836m_e$,) they move much more slowly than electrons, so we are mainly interested in the motion of the {\it electrons}. Hence we can model the motion of the plasma by a single fluid obeying the {\bf compressible Euler equation}.
\end{itemize}

The above leads to the {\bf Euler--Poisson one-fluid} model of the plasma.
Let $e$ denote the elementary charge, and $m$ the mass of the electron.
They are fundamental physical constants whose values are fixed throughout.
The dynamical variables are the electron density $n$, velocity $v$,
pressure $p$, and the electrostatic field $\phi$ which the electrons produce.

By the conservation of charge, $n$ satisfies the {\bf continuity equation}
\begin{equation}\label{cont-eq}
n_t+\nabla\cdot(nv)=0.
\end{equation}
The motion of the electrons satisfies Newton's second law
\begin{equation}\label{Newton2}
nmD_tv=F.
\end{equation}
On the left-hand side of (\ref{Newton2}) is the material derivative of $v$
\[
D_tv=v_t+(v\cdot\nabla)v,
\]
which equals the acceleration of an electron at a given point.
On the right-hand side of (\ref{Newton2}) is the net force acting on that electron
\[
F=-\nabla p+(-ne)(-\nabla\phi)=-\nabla p+en\nabla\phi
\]
where $p$ is the thermal pressure of the electrons, and $\phi$ is the electrostatic potential they produce. Thus we obtain the {\bf Euler equation} for the plasma:
\begin{equation}\label{Euler-eq}
nm(v_t+(v\cdot\nabla)v)=-\nabla p+en\nabla\phi.
\end{equation}
The electrostatic potential $\phi$ is related to $n$ by the {\bf Poisson equation}
\begin{equation}\label{Poisson-eq}
\ep_0\Delta\phi=e(n-n_0)
\end{equation}
where $\ep_0$ is the vacuum permittivity, and $n_0$ is the charge density of the nuclei.

To close the system we need to find the pressure $p$. The ideal gas law says
\begin{equation}\label{gas-law}
p=nk_BT
\end{equation}
where $k_B$ is the Boltzmann constant, and $T$ is also a function of $n$.
The work required to compress the plasma is $-pdV$. By conservation of energy,
this work is converted to the thermal energy of the plasma. Hence
\begin{equation}\label{heat-work}
cnVdT=dQ=-pdV
\end{equation}
where $c$ is the heat capacity of an electron. By the equipartition theorem,
$c=dk_B/2$, where $d$ counts the degree of freedom of the electrons.
Putting this and (\ref{gas-law}) into (\ref{heat-work}) we get
\[
\frac{d}{2}VdT=-TdV
\]
which one can integrate to obtain $\ln V+(d/2)\ln T=$ const,
or $VT^{d/2}=$ const, so $T\propto V^{-2/d}\propto n^{2/d}$.
Again using (\ref{gas-law}) we get
\[
p\propto n^{1+d/2}.
\]
Electrons have no internal degrees of freedom, so $d$ is simply the number of spatial dimensions. In our case $d=2$, and the constitutive equation reads
\begin{align}\label{cnst-eq}
p&=\theta n^2/2,
\end{align}
where $\theta=2k_BT_0/n_0$ depends on the initial condition of the plasma.
We can use (\ref{cnst-eq}) to eliminate the pressure $p$ in (\ref{Euler-eq}).
For simplicity we assume $\theta$ doesn't depend on the position. Then we obtain
\begin{equation}\label{EP-system2}
\begin{aligned}
n_t+\nabla\cdot(nv)&=0,\\
m(v_t+(v\cdot\nabla)v)+\theta\nabla n&=e\nabla\phi,\\
\ep_0\Delta\phi&=e(n-n_0).
\end{aligned}
\end{equation}

\subsection{Conservation laws}
The continuity equation (\ref{cont-eq}) implies the {\bf conservation of charge}
\begin{equation}\label{cnsrv-charge}
\frac{d}{dt} \int n=\int -\nabla\cdot(nv)=0.
\end{equation}
Also conserved is the {\bf energy}
\begin{equation}\label{E-def}
E=\frac{1}{2}\int mnv^2+\ep_0(\nabla\phi)^2+\theta n^2,
\end{equation}
where the three terms on the right correspond to the kinetic energy, electrostatic energy and thermal energy of the plasma respectively.

Next we look at the evolution for the {\bf vorticity} $\omega=\nabla\times v$.
Taking the curl of the second equation in (\ref{EP-system2}) we get its evolution equation
\[
\omega_t=-(v\cdot\nabla)\omega-\omega(\nabla\cdot v).
\]
If the flow is assumed to be {\bf irrotational} (i.e., $\omega=0$) at the beginning,
then it remains so forever.

Under the assumption of irrotationality, we have $(v\cdot\nabla)v=\nabla(|v|^2)/2$.
Then the second equation in (\ref{EP-system2}) shows that $v_t$ is a gradient.
Integrating over the whole torus gives the {\bf conservation of momentum}
\begin{equation}\label{cnsrv-mom}
\frac{d}{dt}\int v=0.
\end{equation}

\subsection{Normalization}
We can rescale the variables to normalize all the constants in the system (\ref{EP-system2}) to 1. To do so, we first list the dimensions of all the physical constants:
\begin{align*}
&\text{Constant} & & m & & e & & n_0 & & \theta & & \ep_0\\
&\text{Dimension} & & MN^{-1} & & CN^{-1} & & NL^{-3} & & L^5T^{-2}MN^{-2} & & L^{-3}T^2M^{-1}C^2,
\end{align*}
where $L$, $T$, $N$, $M$ and $C$ stand for the dimensions of length, time, number (of electrons), mass and charge, respectively. Then we list the dimensions of all the physical variables, and the substitution that makes them dimensionless.
\begin{align*}
&\text{Variable} & & x & & t & & n & & v & & \phi\\
&\text{Dimension} & & L & & T & & NL^{-3} & & LT^{-1} & & L^2T^{-2}MC^{-1}\\
&\text{Substitution} & & \sqrt{\frac{\ep_0\theta}{e^2}}x' & & \sqrt{\frac{\ep_0m}{e^2n_0}}t' & & n_0(1+\rho) & & \sqrt{\frac{n_0\theta}{m}}v' & & \frac{n_0\theta}{e}\phi'.
\end{align*}
Now all the constants are normalized to 1, and the system (\ref{EP-system2}) becomes
\begin{equation}\label{EP-system}
\begin{aligned}
\rho_t+\nabla\cdot((1+\rho)v)&=0,\\
v_t+(v\cdot\nabla)v+\nabla\rho&=\nabla\phi,\\
\Delta\phi&=\rho,
\end{aligned}
\end{equation}

\begin{remark}
There is no more scaling symmetry to be exploited. Hence $R$ is a genuine parameter of the system, and our results depend on $R$ explicitly.
\end{remark}

The quantity $X_0=\sqrt{\ep_0\theta/e^2}$ is called the {\bf Debye length}.
It is the length scale beyond which local fluctuation in charge density
(e.g., near the boundary of the container) does not have a significant effect.
Hence we assume
\[
R=\text{size of torus}/X_0\gg 1.
\]
We also assume the plasma is {\bf charge-neutral}, i.e.,
\[
\int \rho=0
\]
so that the third equation in (\ref{EP-system}) is solvable on the torus.
By a change of reference frame ($\tilde v=v-v_0$, $\tilde\rho(x,t)=\rho(x+v_0t,t)$),
we can also assume the {\bf zero momentum condition}
\[
\int v=0.
\]
By (\ref{cnsrv-mom}), this condition persists for all $t$.

The trivial solution $(\rho,v)=(0,0)$ is an equilibrium of the system (\ref{EP-system}). Our main results in this paper are the long-term stability of this equilibrium.

To linearize the system (\ref{EP-system}), we define the Fourier multipliers $|\nabla|$ and $\Lambda$.
\begin{definition}
Let $\mathcal F$ denote the Fourier transform. We define
\begin{align*}
\mathcal F(|\nabla|u)(\xi)&=|\xi|\mathcal Fu(\xi), &
\mathcal F(\Lambda u)(\xi)&=\sqrt{1+|\xi|^2}\mathcal Fu(\xi).
\end{align*}
\end{definition}

Near this equilibrium the system (\ref{EP-system}) linearizes to
\[
\begin{aligned}
\rho_t&=-\nabla\cdot v,\\
v_t&=-\nabla\rho+\nabla\phi=-\nabla\rho-\nabla|\nabla|^{-2}\rho
=-\nabla(1+|\nabla|^2)|\nabla|^{-2}\rho,
\end{aligned}
\]
which can then be written in matrix form as
\[
\frac{d}{dt}
\begin{pmatrix}
\rho\\
v
\end{pmatrix}
=
\begin{pmatrix}
0 & -\nabla\cdot\\
-\nabla(1+|\nabla|^2)|\nabla|^{-2} & 0
\end{pmatrix}
\begin{pmatrix}
\rho\\
v
\end{pmatrix}
\]
and the eigenvalues of the matrix on the right-hand side are (formally)
\[
\pm\sqrt{\nabla\cdot\nabla(1+|\nabla|^2)|\nabla|^{-2}}
=\pm i\sqrt{1+|\nabla|^2}=\pm i\Lambda.
\]

\subsection{The main theorems}
To state our main results we need to introduce some function spaces.
All functions and integrals are on the torus $(\R/R\Z)^2$ unless stated otherwise.

Let $\varphi$ be a smooth cutoff function that is 1 on $B(0,2/3)$ and vanishes outside $B(0,3/2)$. Let
\begin{align*}
\varphi_j&=\varphi(x/2^j)-\varphi(x/2^{j-1}), &
\varphi_{\le j}&=\varphi(x/2^j).
\end{align*}
Let $P_k$ be the Littlewood-Paley projection onto frequency $2^k$, so that
\begin{align*}
\mathcal F(P_ku)&=\varphi_k \mathcal Fu, &
\mathcal F(P_{\le k}u)&=\varphi_{\le k} \mathcal Fu, &
\mathcal F(P_{>k}u)&=(1-\varphi_{\le k})\mathcal Fu.
\end{align*}
For $j\ge1$ let $Q_j$ be the physical localization at scale $\approx 2^j$,
that is, multiplication by $\varphi_j$. Let $Q_0=id-\sum_{j\ge1}Q_j$.
Define $\|x\|=d(x,(R\Z)^2)$, and $k^+=\max(k,0)$. Fix $M\ge10$.
Now we can introduce the function spaces.

\begin{definition}\label{XZ-def}
Define
\begin{align*}
\|u\|_X&=\sum_{k\in\Z} 2^{Mk^+}\|P_ku\|_{L^\infty}, &
\|u\|_Z&=\|(1+\|x\|)^{2/3}\Lambda^{M+2}u\|_{L^2}.
\end{align*}
\end{definition}

One can think of the $X$ norm as $W^{M,\infty}$, and the $Z$ norm as $W^{M+2,1.2+}$.

\begin{lemma}\label{XZ-bound}
(i) For $k\in\Z$ we have
\[
\|P_ku\|_{W^{M+2,1.2+}}\lesssim \|u\|_Z,
\]
where $a+$ denotes an exponent larger than $a$ but can be arbitrarily close to $a$.

(ii)
\begin{align*}
\|u\|_X&\lesssim \sum_{k\in\Z} 2^{k+Mk^+}\|P_ku\|_{L^2}\lesssim \|u\|_{H^{M+2}}, &
\|u\|_Z&\lesssim R^{2/3}\|u\|_{H^{M+2}}.
\end{align*}

(iii) If $m\ge0$ and $T$ is a differential operator whose symbol is in the class $S^m_{1,0}$, then
\[
\|Tu\|_{W^{M-m,\infty}}\lesssim \|u\|_X.
\]

(iv) For $k\in\Z$ we have
\[
\|(1+\|x\|)^{2/3}P_ku\|_{L^2}\lesssim 2^{-(M+2)k^+}\|u\|_Z.
\]

(v) Calderon--Zygmund operators are bounded on $H^N$, $X$ and $Z$.
\end{lemma}
\begin{proof}
(i) follows from H\"older's inequality and the fact that $(1+\|x\|)^{-2/3}\in L^{3+}$. (ii) follows from the Bernstein inequality. (iii) follows from the bound $\|TP_ku\|_{L^\infty}\lesssim\|P_ku\|_{L^\infty}$.
(iv) and (v) are due to the fact that the weight $(1+\|x\|)^{4/3}\in A_2$, see Theorem 2 and Corollary in Section V.4 of \cite{St}.
\end{proof}
\begin{remark}
Throughout the paper $N$ and $M$ are fixed integers, and all implicit constants only depend on them unless stated otherwise. In particular they do not depend on $R$.
We assume $M\ge10$, $\ep$ is small enough and $R$ is large enough.
\end{remark}

Now we can state the main results about the 2D Euler--Poisson system (\ref{EP-system}) in this paper.
\begin{theorem}\label{Thm1}
There is a constant $c>0$ such that if $N\ge M+5$, $1\ll R\le \exp(c\ep^{-c})$ and
\begin{equation}\label{HN0}
\||\nabla|^{-1}\rho_0\|_{H^{N+1}}+\|v_0\|_{H^N}\le\ep,
\end{equation}
then there is
\begin{equation}\label{T1}
T_{R,\ep}\approx R/(\ep^2(\log R)^{O(1)})
\end{equation}
such that (\ref{EP-system}) with initial data $(\rho_0,v_0)$,
has a unique solution with $(\Lambda|\nabla|^{-1}\rho,v)\in C([0,T_{R,\ep}],H^N)$.
\end{theorem}

\begin{theorem}\label{Thm2}
If $N\ge\max(3(M+4),106)$, $R\gg1$ and
\begin{equation}\label{Z0}
\||\nabla|^{-1}\rho_0\|_{H^{N+1}}+\|v_0\|_{H^N}+\|\Lambda|\nabla|^{-1}\rho_0\|_Z+\|v_0\|_Z\le\ep,
\end{equation}
then there is
\begin{equation}\label{T2}
T_{R,\ep}'\approx R^{10/9-O(1/N)}\ep^{-2/3+O(1/N)}
\end{equation}
such that (\ref{EP-system}) with initial data $(\rho_0,v_0)$,
has a unique solution with $(\Lambda|\nabla|^{-1}\rho,v)\in C([0,T_{R,\ep}],H^N)$.
\end{theorem}

\begin{remark}\label{Alx0}
(i) For specific values of the constants in exponents see Propositions \ref{bootstrap1} and \ref{bootstrap2}. Technically the proof gives a constant $c$ depending on the choice of $N$ and $M$ in Theorems \ref{Thm1} and \ref{Thm2}, but this dependence can be removed using persistence of regularity (see \cite{Ta} for details.)

(ii) If the assumption made on $R$ and $\ep$ in Theorem \ref{Thm1} does not hold, then we have $R\gtrsim\ep^{-100}$ (say), and Theorem \ref{Thm2} gives a better bound. Either lifespan is longer than $R/\ep$, which is the most one can hope for without using the normal form, for no decay can be expected of the $L^2$ norm of the solution, giving a lower bound of $\gtrsim\ep/R$ of its $L^\infty$ norm.

(iii) Theorem \ref{Thm2} implies global regularity in the Euclidean case, and gives a quantitative version of the theorem of Ionescu--Pausader \cite{IoPa1} and Li--Wu \cite{LiWu}. A notable difference from those two works is that here the spatial weights on the Sobolev norms are equivalent to $x^{2/3}$ in the Euclidean case, which is smaller than $x$ in \cite{IoPa1,LiWu}. This is due to quartic energy estimates which allow for more flexibility in choosing the spatial weights. The choice of $x^{2/3}$ as spatial weights will be explained later when the $Z$-norm estimates are discussed.

(iv) The Klein-Gordon equation with mass $m=1$ on a torus of size $R$ can be rescaled to the unit torus, but with mass $m=R$. When $R$ is of unit size, Theorem \ref{Thm1} gives a lifespan $\gtrsim\ep^{-2}$ and recovers the theorem of Delort--Szeftel \cite{DeSz}.
In general, our bounds depend on $R$ in a uniform way, thus reinstating the exceptional set of measure zero that has to be excluded from the parameter space in Delort \cite{De1} and Fang--Zhang \cite{FaZh}.

(v) Unlike Faou-Germain-Hani \cite{FaGeHa} and Buckmaster--Germain--Hani--Shatah \cite{BuGeHaSh}, our proofs of Theorems \ref{Thm1} and \ref{Thm2} do not rely on the number-theoretic properties of the resonance set. Hence it is straightforward to generalize our results to nonsquare tori with bounded aspect ratios.
\end{remark}

\subsection{Previous work on long-term regularity}
The Euler--Poisson system (\ref{EP-system}) is a symmetrizable quasilinear hyperbolic system, as already shown in \cite{IoPa1}. Therefore local regularity of solutions with sufficiently smooth initial data follows from \cite{Ka}. It is long-term regularity that is of interest here.

In the Euclidean case, global regularity in 3D was shown by Guo in \cite{Guo}, and in 2D shown independently by Ionescu--Pausader \cite{IoPa1} and Li--Wu \cite{LiWu}. Extensions of this model whose global regularity is known include the nonneutral case \cite{GeMaPa}, the Euler--Maxwell equation (\cite{GeMa,IoPa2} in 3D and \cite{DeIoPa} in 2D), the Euler--Poisson ion equation \cite{GuPa}, two fluid models (\cite{GuIoPa} for nonrelativistic models and \cite{GuIoPa2} for relativistic models), general Klein-Gordon systems (\cite{DeFaXu,Ge,IoPa2} for generic parameters and \cite{Deng} for all parameters) and coupled wave-Klein-Gordon systems \cite{IoPa3}.

\subsubsection{Periodic solutions}
Because of the lack of dispersion on compact domains,
global regularity is hard to come by in the periodic case.
The study of periodic dispersive equations was initiated by Bourgain,
who showed almost global regularity of the quadratic Klein-Gordon equation for almost every mass on the circle \cite{Bo2}. Using the {\bf normal form} method, this result has since been generalized to semilinear \cite{De1,FaZh} and quasilinear \cite{De2,DeSz} Klein-Gordon equations on tori, spheres \cite{De3} and Zoll manifolds \cite{DeSz2}.

It should also be mentioned that when the size of the torus is large (known as the {\bf large box limit}), Faou--Germain--Hani \cite{FaGeHa}, and more recently Buckmaster--Germain--Hani--Shatah \cite{BuGeHaSh}, were able to derive a continuous resonance equation that describes the long-term behavior of the solution of a cubic Schr\"{o}dinger equation on the torus of any dimension.

\subsection{Main ideas of the proof}\label{MainIdea}

Since the seminal work of Klainerman \cite{Kl2}--\cite{Kl4}, Christodoulou \cite{Chr}, and Shatah \cite{Sh}, the proof of long-term regularity of such systems consists of the following two aspects:

\setlength{\leftmargini}{1.8em}
\begin{itemize}
  \item[(1)] Energy estimates (high order Sobolev norms) to control high frequencies;
\smallskip
  \item[(2)] Dispersive estimates of the $L^\infty$ norm of the solution to control low frequencies.
\end{itemize}

Starting from Shatah \cite{Sh}, Poincar\'e's {\it normal form} method (see \cite{Ar,Ch} for book reference) has proved to be successful in the study of long-term solutions of nonlinear evolutions. Basically one transforms quadratic nonlinearities to {\it cubic} ones to gain better integrability of the decay of the solution. To adapt this general framework to the case of the torus, one needs to overcome the difficulty that the $1/t$ decay of the linear evolution of the Euclidean Klein-Gordon equation is only valid for time $t\lesssim R$. Beyond this time the solution wraps around and superimposes with itself. As a rough estimate, notice that the group velocity of the Klein-Gordon wave is $\nabla\Lambda(\xi)=O(1)$, so after time $t\gtrsim R$ the solution is able to wrap around the torus $O(t/R)$ times, both horizontally and vertically. Thus the Euclidean theory only gives a bound of the form
\begin{equation}\label{u-Loo}
\|u\|_{L^\infty}\lesssim t^{-1}(t/R)^2=t/R^2
\end{equation}
on the torus of size $R$. Suppose the nonlinearity has degree $D$. Using Gronwall's inequality one arrives at an energy estimate schematically of the form
\begin{equation}\label{E-estimate}
E(t)\le E(0)\exp\left( C\int_0^t \|u(s)\|_{L^\infty}^{D-1} \right)ds
\le E(0)\exp\left( C\frac{t^D\ep^{D-1}}{R^{2D-2}} \right).
\end{equation}
Hence one is only able to close the estimate up to time
\begin{equation}\label{crude-t-bound}
T\lesssim R^{\frac{2D-2}{D}}\ep^{-\frac{D-1}{D}}.
\end{equation}
For quadratic nonlinearity this gives a lifespan of $R/\sqrt\ep$. This is nontrivial only in the large box limit $R\to\infty$; when $R\approx1$, this lifespan is even shorter than provided by local regularity.

To improve on (\ref{crude-t-bound}), we will combine the following three ingredients;

\begin{itemize}
\item {\bf Quartic energy estimates}. It has been observed in \cite{Sh,IoPa1,LiWu} that the Klein-Gordon equation has no time resonance, and the normal form transform effectively makes the nonlinearity {\it cubic}, and allows for {\it quartic} energy estimates. To overcome the loss of derivatives arising from quasilinearity, we make use of paradifferential calculus, which has already found application to similar quasilinear evolution equations in, for example, \cite{DeIoPa}.

\item {\bf $Z$-norm estimates} in the large box limit. When $R$ is very large compared to $1/\ep$, the dispersive estimates are done using a bootstrap argument in a suitable $Z$-norm (spatially weighted Sobolev norm) of the profile. The argument is similar to the Euclidean case in \cite{IoPa1}. Thanks to the normal form transform, the nonlinearity is now cubic, so we only need a decay better than $1/\sqrt t$. We will still optimize the Euclidean decay rate, which translates to longer lifespan in the large box case. This is done by adjusting the spatial weight.

More precisely, in the Euclidean case, a spatial weight of $x^\alpha$,
where $\alpha\in(0,1)$, leads to $t^{-\alpha}$ decay of the $L^\infty$ norm of the solution, and no decay of the $L^2$ norm.
Localized initial data will spread a distance of $\approx t$ after time $t$. At such distances, the $Z$ norm of the cubic nonlinearity is dominated by two scenarios: one where all three factors are at distance $O(1)$ from the origin, and the other where all three factors are at distance $\approx t$. Using the $L^2\times L^\infty\times L^\infty\to L^2$ trilinear estimate, the first part decays like $t^{-2}$, while the second one decays like $t^{-3\alpha}$ (note that we only need the {\it unweighted} $L^2$ norm of the first factor, which contributes another factor $t^{-\alpha}$ in decay.) Putting the weight back onto the nonlinearity gives a decay rate of $\max(t^{-2+\alpha},t^{-2\alpha})$, which attains the minimum of $t^{-4/3}$ when $\alpha=2/3$. This is well integrable, and translates to the fact that the exponent of $R$ in the lifespan is larger than 1 in the large box case.

\item {\bf Strichartz estimates} in the small box regime. When $R$ is close to 1, the decay (\ref{u-Loo}) is no longer useful; in fact we don't expect any decay of the $L^\infty$ norm of the linear evolution. Trivial integration will produce a factor of $t$ in the energy estimate. This factor can be saved using Strichartz estimates, which in $\R^2$ reads (see \cite{Ta} for a textbook reference)
\begin{equation}\label{Strichartz}
\|e^{it\Delta}u\|_{L^q_tL^r_x}\lesssim \|u\|_{L^2_x},\quad 1/q+1/r=1/2,\ (q,r)\neq(2,\infty).
\end{equation}
We will show its analog for the Klein-Gordon equation on $\mathbb T^2$ at the endpoint $(q,r)=(2,\infty)$. This fits nicely into the energy estimate (\ref{E-estimate}) thanks again to the cubic nonlinearity (so $D-1=2$). The logarithmic loss in time in the endpoint case of (\ref{Strichartz}) is reflected in (\ref{T1}), and the loss of derivative in quasilinearity is recovered by the energy estimates.
\end{itemize}

\subsection{Organization}
The rest of the paper is organized as follows: In section \ref{Local} we establish local wellposedness of the Euler-Poisson system and state the main bootstrap propositions. In section \ref{LinEst} we introduce paradifferential calculus and derive the linear dispersive estimates and multilinear paraproduct estimates to be used in the rest of the paper. In section \ref{EneEst} we obtain the quartic energy estimates. In section \ref{StrEst} we prove Theorem \ref{Thm1}, and in section \ref{ZEst} we prove Theorem \ref{Thm2}.

\subsection{Acknowledgements}
The author wants to thank his advisor Alexandru Ionescu for his constant help and unfailing encouragement throughout the completion of this work. He also wishes to thank Xuecheng Wang, Yu Deng and Ziquan Zhuang for helpful discussions, and Qingtang Su for numerous illuminating questions that improved the exposition greatly.

\section{Local wellposedness and bootstrap propositions}\label{Local}
By the irrotationality and zero momentum conditions, we have $v=\nabla h$ for some $h$.
Let $g=|\nabla|^{-1}\rho$, and
\[
U=\Lambda g+i|\nabla|h=X+iY.
\]
The charge-neutrality and zero momentum conditions imply that $\mathcal FU(0)=0$,
and the evolution equations for $X$ and $Y$ are
\begin{equation}\label{XtYt}
\begin{aligned}
X_t&=\Lambda Y-\Lambda R_j(\rho v_j), &
Y_t&=-\Lambda X-|\nabla|(v^2/2),
\end{aligned}
\end{equation}
where $\rho=\Lambda^{-1}|\nabla|X$, $v_j=R_jY$, and $R_j=|\nabla|^{-1}\partial_j$ is the Riesz transform.
Note that the action of the Riesz transform on the zero frequency is assumed to be 0, consistent with the charge-neutral condition.
Repeated indices imply Einsteinian summation throughout the paper.

\subsection{Local Wellposedness}
The local wellposedness of the Euler-Poisson system (\ref{EP-system}) on $\R^2$ was worked out in Proposition 2.2 (i) of \cite{IoPa1} using Bona-Smith approximation.
The key point in the proof is an energy identity making use of the symmetrizability of the Euler-Poisson system, which is manifest after multiplying the second equation by $n$.
Since the proof of the energy identity uses nothing more than integration by parts,
it continues to hold on the torus, which then shows local wellposedness of the Euler-Poisson system on the torus.

\begin{proposition}\label{local-existence}
If $N\ge3$, $U_0\in H^N$ and $\|U_0\|_{H^3}$ is sufficiently small,
then there is $U\in C([0,1],H^N)\cap C^1([0,1],H^{N-1})$ solving (\ref{EP-system}) with initial data $U_0$.
\end{proposition}
It follows from Proposition \ref{local-existence} and Lemma \ref{XZ-bound} (ii) that
\begin{proposition}\label{XZ-fin-cont}
If $N\ge M+2$, $T>0$ and $\sup_{t\in[0,T]}\|U(t)\|_{H^N}$ is sufficiently small,
then $\|U(t)\|_X$ and $\|V(t)\|_Z$ are finite and continuous on $[0,T]$.
\end{proposition}

\subsection{Bootstrap propositions}
In this subsection we lay out the bootstrap propositions and use them to show Theorem \ref{Thm1} and Theorem \ref{Thm2}.
Thoughout the paper we assume $R\gg 1$ and put
\begin{align*}
\mathcal L&=\log(t+1), & \mathcal L_R&=\log R.
\end{align*}
We will use $\mathcal L$ in Section \ref{ZEst} and $\mathcal L_R$ in Section \ref{LinEst} and Section \ref{StrEst}.

\begin{proposition}\label{bootstrap1}
Fix $N\ge M+5\ge 15$. Assume (\ref{HN0}) holds with $\ep$ small enough. Also assume
\begin{equation}\label{growthX1}
\begin{aligned}
\|U\|_{L^\infty([0,t])H^N}&\le\ep_1,\\ \|U\|_{L^2([0,t])X}&\le\ep_2,
\end{aligned}
\end{equation}
with $\ep_1$, $\ep_2$ small enough. Then
\begin{align}
\label{growth-Em-X}
\|U\|_{L^\infty([0,t])H^N}&\lesssim \ep+\ep_1^{3/2}+\ep_1\ep_2,\\
\label{growthX2}
\|U\|_{L^2([0,t])X}&\lesssim \mathcal L_R\sqrt{1+t/R}\cdot\ep_1(1+\mathcal L_R^{3/2}\ep_2^2).
\end{align}
\end{proposition}

\begin{proof}[Proof of Theorem \ref{Thm1}]
We can choose $\ep_1\approx\ep$, $T_0\approx R/(\ep^2(\log R)^{7/2})$
and $\ep_2\approx(\log R)^{-3/4}$ such that for $t\le T_0$,
(\ref{growth-Em-X}) and (\ref{growthX2}) give (\ref{growthX1}) with the strict inequality.
Note that we need $R\le\exp(c\ep^{-4/7})$ to recover the $L^2X$ norm.

Let
\[
T^*=\sup\{t\le T_0:\exists\text{ a solution }U\in C([0,t],H^N)\text{ such that }(\ref{growthX1})\text{ holds}\}.
\]
Then (\ref{growthX1}) holds for $T<T^*$. By Proposition \ref{local-existence},
the solution can be extended to $C([0,T^*+1/2],H^N)$.
If $T^*<T_0$, then by Proposition \ref{XZ-fin-cont},
(\ref{growthX1}) holds beyond time $T^*$, contradicting the definition of $T^*$.
Hence $T^*\ge T_0$, and the solution exists up to time $T_0$.
\end{proof}

\begin{proposition}\label{bootstrap2}
Fix $N\ge\max(3(M+4),106)$. Assume (\ref{Z0}) holds with $\ep$ small enough. Define the profile $V(t)=e^{it\Lambda}U(t)$. Assume
\begin{equation}\label{growthZ}
\begin{aligned}
\|U\|_{L^\infty([0,t])H^N}&\le\ep_1,\\
\sup_{[0,t]} \|V\|_Z&\le\ep_1,
\end{aligned}
\end{equation}
with $\ep_1$ small enough. Then
\begin{align}
\label{growth-Em-Z}
\|U\|_{L^\infty([0,t])H^N}
&\lesssim \ep+\ep_1^{3/2}+t^{6/5}\ep_1^2/R^{4/3},\\
\label{growthZ2}
\sup_{[0,t]}\|V\|_Z
&\lesssim \ep+(t^{1+9/N}/R^{4/3}+1)\ep_1^2+(t^{3+33/N}/R^{10/3-2/N}+1)\ep_1^3.
\end{align}
\end{proposition}

\begin{proof}[Proof of Theorem \ref{Thm2}]
We can choose $\ep_1\approx\ep$ and $T_0\approx R^{\frac{10/3-2/N}{3+33/N}}\ep^{-\frac{2}{3+33/N}}$ such that for $t\le T_0$, (\ref{growth-Em-Z}) and (\ref{growthZ2}) lead to (\ref{growthZ}) with the strict inequality. Then the proof is similar to that of Theorem \ref{Thm1}.
\end{proof}

\section{Linear dispersive and multilinear paraproduct estimates}\label{LinEst}
\subsection{Linear dispersive estimates}
The first ingredient in the proof of global existence is dispersive estimates.
In the following we will use the fact that $P_{<-\mathcal L}=0$ unless $t\lesssim R$.

\begin{lemma}\label{dispersive}
For $k\in\Z$, $1\le p\le q\le\infty$ with $1/p+1/q=1$ we have
\[
\|P_ke^{-it\Lambda}u\|_{L^q}\lesssim [(1+t)^{-1}(t/R+1)^22^{2k^+}]^{1/p-1/q}\|u\|_{L^p}.
\]
\end{lemma}
\begin{proof}
By interpolation and unitarity of $e^{-it\Lambda}$ we can assume $p=1$ and $q=\infty$.
Since $P_k=P_kP_{\le0}$ for $k<0$, and $P_k$ is bounded on $L^\infty$, we can further assume $k\in\N$ or $k$ is ``$\le0$".
By Poisson summation,
\begin{align*}
P_ke^{-it\Lambda}u(x)
&=\frac{C}{R^2}\sum_{\xi\in(2\pi\Z/R)^2} e^{ix\cdot\xi-it\Lambda(\xi)}\varphi_k(\xi)\hat u(\xi)
=\int G_k(x,y,t)u(y)dy,\\
G_k(x,y,t)&=\sum_{z\in(R\Z)^2} \mathcal K_k(x,y+z,t),\\
\mathcal K_k(x,y,t)&=\int e^{i(x-y)\cdot\xi-it\Lambda(\xi)}\varphi_k(\xi)d\xi.
\end{align*}
Trivially $|\mathcal K_k(x,y+z,t)|\lesssim2^{2k}$.
When $t\ge1$, we can get a better bound
using the method of stationary phase. The gradient of the phase is
\[
\nabla((x-y-z)\cdot\xi-t\Lambda(\xi))=x-y-z-t\xi/\Lambda(\xi).
\]
It vanishes only when $|x-y-z|\le2t$, which happens for $O((t/R+1)^2)$ values of $z$. For such $z$, in a coordinate whose $x$-axis points in the direction of $\xi$, the Hessian of the phase is
\[
-t
\begin{pmatrix}
\Lambda''(\xi) & 0\\
0 & \Lambda'(\xi)/|\xi|
\end{pmatrix}
=-t
\begin{pmatrix}
1/\Lambda(\xi)^3 & 0\\
0 & 1/\Lambda(\xi)
\end{pmatrix},
\]
with nonvanishing determinant $t^2/\Lambda(\xi)^4$. Then by stationary phase,
\[
|\mathcal K_k(x,y+z,t)|\lesssim (t^2/\Lambda(\xi)^4)^{-1/2}\lesssim t^{-1}2^{2k}.
\]
For other values of $z$, the gradient of the phase is $\gtrsim|x-y-z|$, so
\[
\left| \sum_{z\in(R\Z)^2\atop|x-y-z|>2t} \mathcal K_k(x,y+z,t) \right|
\lesssim 2^{2k}\sum_{z\in(R\Z)^2\atop|x-y-z|>2t}|x-y-z|^{-3}
\lesssim t^{-1}2^{2k}.
\]
Combining the two bounds shows the claim.
\end{proof}

\begin{lemma}\label{dispersiveZ}
(i) For $k\in\Z$ we have
\[
\|P_ke^{-it\Lambda}u\|_{L^\infty}
\lesssim 2^{k/3-(M+2/3)k^+}(1+t)^{-\frac23+}(t/R+1)^{4/3}\|u\|_Z.
\]

(ii)
\[
\|e^{-it\Lambda}u\|_X\lesssim (1+t)^{-\frac23+}(t/R+1)^{4/3}\|u\|_Z.
\]
\end{lemma}
\begin{proof}
By the Bernstein inequality and Lemma \ref{dispersive}, for any $c\in(1/3,1/2)$,
\begin{align*}
\|P_ke^{-it\Lambda}u\|_{L^\infty}&\lesssim 2^{ck}\|P_ke^{-it\Lambda}u\|_{L^{2/c}}\\
&\lesssim 2^{ck+2(1-c)k^+}(1+t)^{-1+c}(t/R+1)^{2(1-c)}\|P_{[k-1,k+1]}u\|_{L^{2/(2-c)}}.
\end{align*}
Then (i) follows from Lemma \ref{XZ-bound} (i).
To get (ii) we sum (i) over $k\in\Z$.
\end{proof}

\begin{lemma}\label{dispersiveTT*}
\[
\|e^{-is\Lambda}u\|_{L^2([0,t])X}\lesssim \mathcal L_R\sqrt{1+t/R}\|u\|_{H^{M+2}}.
\]
\end{lemma}
\begin{proof}
For $k\in\Z$ let
\begin{align*}
T_k&:L_x^2\to L^2([0,t])L_x^\infty, & u&\mapsto P_ke^{-is\Lambda}u.
\end{align*}
Then
\begin{align*}
T_k^*&:L^2([0,t])L_x^1\to L_x^2, & u&\mapsto \int_0^t P_ke^{is\Lambda}u(s)ds.
\end{align*}
Then $\|T_k\|=\|T_k^*\|=\|T_kT_k^*\|^{1/2}$, where
\begin{align*}
T_kT_k^*&:L^2([0,t])L_x^1\to L^2([0,t])L_x^\infty, & u&\mapsto \int_0^t P_k^2e^{i(s'-s)\Lambda}u(s')ds'.
\end{align*}

First we suppose $t\le R$. By Lemma \ref{dispersive} (note that $t/R+1\le2$),
\[
\|P_k^2e^{i(s'-s)\Lambda}u(s')\|_{L_x^\infty}\lesssim (1+|s'-s|)^{-1}2^{2k^+}\|u(s)\|_{L_x^1}.
\]
Then by Young's inequality (note that now $\mathcal L_R>\log(1+t)$),
\[
\|T_kT_k^*u\|_{L^2([0,t])L_x^\infty}\lesssim \|(1+|\cdot|)^{-1}\|_{L^1([0,t])}2^{2k^+}\|u\|_{L^2([0,t])L_x^1}<2^{2k^+}\mathcal L_R\|u\|_{L^2([0,t])L^1_x}
\]
so $\|T_k\|\lesssim 2^{k^+}\sqrt{\mathcal L_R}$. For $t>R$ we use the unitarity of $e^{it\Lambda}$ to take an $\ell^2$ sum of time intervals of length $R$ to get $\|T_k\|\lesssim 2^{k^+}\sqrt{\mathcal L_R(1+t/R)}$.

Applying the above bound to $P_{[k-1,k+1]}u$ and summing over $k\ge-\mathcal L_R$, using the Cauchy--Schwarz inequality on the right-hand side,
we get the desired bound.
\end{proof}

\subsection{Paradifferential calculus}
We will use Weyl quantization of paradifferential operators on the torus.
\begin{definition}
Given a symbol $a=a(x,\zeta): (\R/R\Z)^2\times(\R^2\backslash0)\to\C$, define the operator $T_a$ using the following recipe:
\begin{equation}\label{para-def}
\mathcal F(T_af)(\xi)=\frac{C}{R^2}\sum_{\eta\in(2\pi\Z/R)^2}
\varphi_{\le-10}\left( \frac{|\xi-\eta|}{|\xi+\eta|} \right)\mathcal F_x a\left( \xi-\eta,\frac{\xi+\eta}2 \right)\hat f(\eta),
\end{equation}
where $C$ is a normalization constant (independent of $R$) such that $T_1=\text{id}$.
\end{definition}
\begin{remark}
With the inclusion of the factor $\varphi_{\le-10}$, only low frequencies of $a$ and high frequencies of $f$ are involved in $T_af$.
When $\xi+\eta=0$, this factor is taken as 0, so $a(x,0)$ will never be used. Also the case $\xi=\eta=0$ is of no concern, for $T_a$ will only act on $U$, for which we have assumed $\mathcal FU(0)=0$.
\end{remark}

The next lemma follows directly from the definition.

\begin{lemma}\label{para2diff}
(i) If $a$ is real valued, then $T_a$ is self-adjoint.

(ii) If $a(x,-\zeta)=\overline{a(x,\zeta)}$ and $f$ is real valued, so is $T_af$.

(iii) If $a=P(\zeta)$, then $T_af=P(D)f$ is a Fourier multiplier.

(iv) For $k\in\Z$ we have $P_kT_a(P_{\le k-2}f)=0$.

(v) If $a=a(x)$, then for $k\in\Z$ we have $P_kT_{P_{\le k-20}a}f=P_k(P_{\le k-20}a\cdot f)$.
\end{lemma}

The following symbol norm will be used.
\begin{definition}
For $p\in[1,\infty]$ and $m\in\R$ define
\[
\begin{aligned}
|a|(x,\zeta)
&=\sum_{|I|\le8} |\zeta|^{|I|}|\partial_{\zeta_I} a(x,\zeta)|, &
\|a\|_{\mathcal L_m^p}
&=\sup_{\zeta\in\R^2\backslash0} (1+|\zeta|)^{-m}\||a|(x,\zeta)\|_{L^p_x(\R/R\Z)^2}.
\end{aligned}
\]
\end{definition}

Here $m$ is the {\it order} of the symbol, in the sense of  H\"ormander.
\begin{lemma}\label{Sm-Lm}
A multiplier whose symbol is of class $S^m_{1,0}$ has finite $\mathcal L_m^\infty$ norm.
\end{lemma}

For functions independent of $\zeta$, the $\mathcal L_m^p$ norm agrees with the $L^p$ norm.
\begin{lemma}\label{Lmq=Lq}
If $a=a(x)$ and $m\ge0$, then $\|a\|_{\mathcal L_m^p}=\|a\|_{L^p}$.
\end{lemma}

The norm of the product of two operators can be bounded using Leibniz's rule and H\"older's inequality. The result is
\begin{lemma}\label{prod-Lm}
For fixed $m\in\R$, $p,q,r\in[1,\infty]$ with $1/p=1/q+1/r$ we have
\[
\|ab\|_{\mathcal L_{m+n}^p}\lesssim \|a\|_{\mathcal L_m^q}\|b\|_{\mathcal L_n^r}.
\]
\end{lemma}

Paradifferential operators in $\mathcal L_m^q$ act like differential operators of order $m$ with $L^q$ coefficients.
\begin{lemma}\label{Taf-Lp}
For fixed $m\in\R$, $p,q,r\in[1,\infty]$ with $1/p=1/q+1/r$ we have
\[
\|P_kT_af\|_{L^p}\lesssim 2^{mk^+}\|a\|_{\mathcal L_m^q}\|P_{[k-2,k+2]}f\|_{L^r}.
\]
\end{lemma}
\begin{proof}
First we assume that $k\ge-\log R-1$; otherwise $P_kf=0$ for any function $f$ on $(\R/R\Z)^2$ because $\supp\varphi_k$ and $(\Z/R)^2$ are disjoint.

The Schwartz kernel for $P_kT_a$ is
\begin{align*}
S(x,y)&=\frac{C'}{R^4}\sum_{\xi,\eta\in(2\pi\Z/R)^2} \int a\left( z,\frac{\xi+\eta}2 \right)e^{i(\xi\cdot(x-z)+\eta\cdot(z-y))}\varphi_{\le-10} \left( \frac{|\xi-\eta|}{|\xi+\eta|} \right)\varphi_k(\xi) dz\\
&=\int \frac{C'}{R^4}\sum_{\xi,\eta\in(2\pi\Z/R)^2} a\left( z,\frac{\xi+\eta}2 \right)\varphi_{\le-10}\left( \frac{|\xi-\eta|}{|\xi+\eta|} \right)\varphi_k(\xi)\\
&\times\frac{\left( \prod_{j=1}^{|I|}\Delta_{1/R}^{\xi_{I_j}} \right)
\left( \prod_{j=1}^{|J|}\Delta_{1/R}^{\eta_{J_j}} \right)
e^{i(\xi\cdot(x-z)+\eta\cdot(z-y))}}{R^{|I|+|J|}
\prod_{j=1}^{|I|}(e^{2\pi i(x_{I_j}-z_{I_j})/R}-1)
\prod_{j=1}^{|J|}(e^{2\pi i(z_{J_j}-y_{J_j})/R}-1)}dz.
\end{align*}
Here
\[
\Delta_h^{\zeta_j} a=\frac{a(x,\zeta+h\mathbf e_j)-a(x,\zeta)}{h}
\]
is the finite difference quotient. By summation by parts in $\xi$ and $\eta$ it follows that
\begin{align}
\nonumber
S(x,y)&=\int \frac{C'}{R^4}\sum_{\xi,\eta\in(2\pi\Z/R)^2}
e^{i(\xi\cdot(x-z)+\eta\cdot(z-y))}\\
\label{diff-quot}
&\times\frac{\left( \prod_{j=1}^{|I|}\Delta_{1/R}^{\xi_{I_j}} \right)
\left( \prod_{j=1}^{|J|}\Delta_{1/R}^{\eta_{J_j}} \right)
\left( a\left( z,\frac{\xi+\eta}2 \right)\varphi_{\le-10}\left( \frac{|\xi-\eta|}{|\xi+\eta|} \right)\varphi_k(\xi) \right)}
{R^{|I|+|J|}\prod_{j=1}^{|I|}(e^{2\pi i(x_{I_j}-z_{I_j})/R}-1)
\prod_{j=1}^{|J|}(e^{2\pi i(z_{J_j}-y_{J_j})/R}-1)}dz.
\end{align}

By the fundamental theorem of calculus, the Leibniz rule, the bounds
$|\nabla^l\varphi_k|$, $\left| \nabla^l\varphi_{\le-10}\left( \frac{|\xi-\eta|}{|\xi+\eta|} \right) \right|\lesssim_l 2^{-lk}$ and the triangle inequality, the difference quotient in (\ref{diff-quot}) can be bounded by $C_{|I|+|J|}$ times
\begin{align*}
g_{I,J}(z,\xi,\eta)&=\varphi_{[k-1,k+1]}\left( \frac{\xi+\eta}{2} \right)
\int_{[0,1]^{|I|+|J|}} \sum_{l=0}^{|I|+|J|} 2^{-(|I|+|J|-l)k}\\
&\times\left| \nabla_\zeta^la\left( z,\frac{\xi+\eta+(\sum_{j=1}^{|I|}t_j\mathbf e_{I_j}+\sum_{j=1}^{|J|} t_j\mathbf e_{J_j})/R}2 \right) \right|dt_1\cdots dt_{|I|+|J|}.
\end{align*}

When $k\ge-\log R+O(1)$, the second argument of $a$ is still $\asymp 2^k$,
it follows from the definition of $|a|$ that for $|I|+|J|\le5$,
\[
g_{I,J}(z,\xi,\eta)\lesssim 2^{-(|I|+|J|)k} \int_{[0,1]^{|I|+|J|}} 
|a|\left( z,\frac{\xi+\eta+(\sum_{j=1}^{|I|}
t_j\mathbf e_{I_j}+\sum_{j=1}^{|J|} t_j\mathbf e_{J_j})/R}2 \right)dt_1\cdots dt_{|I|+|J|}.
\]
Using $|e^{2\pi ix/R}-1|\approx\|x\|/R$ and $\|x-z\|+\|z-y\|\approx\|x-y\|+\|x-z\|$, where $\|x\|=d(x,R\Z)$, we have
\begin{align*}
|S(x,y)|&\lesssim \min_{|I|+|J|\le5} \int R^{-4} \sum_{\xi,\eta\in(2\pi\Z/R)^2\atop|\xi|,|\eta|\approx 2^k} \frac{g_{I,J}(z,\xi,\eta)}
{\prod_{j=1}^{|I|} \|x_{I_j}-z_{I_j}\|\prod_{j=1}^{|J|} \|z_{J_j}-y_{J_j}\|}dz\\
&\lesssim \min_{|I|+|J|\le5} \int R^{-4} \sum_{\xi,\eta\in(2\pi\Z/R)^2\atop|\xi|,|\eta|\approx 2^k} \frac{g_{I,J}(z,\xi,\eta)}
{\prod_{j=1}^{|I|} \|x_{I_j}-y_{I_j}\|\prod_{j=1}^{|J|} \|x_{J_j}-z_{J_j}\|}dz.
\end{align*}
Therefore
\[
|P_kT_af(x)|\lesssim \min_{|I|+|J|\le5} \int \mathcal K(y,z)f(x-y)\tilde g(x-z)dydz,
\]
where
\begin{align*}
\mathcal K(y,z)&=\frac{1}{(1+2^k(\|y\|+\|z\|))^5}, &
\tilde g(z)&=\max_{|I|+|J|\le5} \frac{2^{(|I|+|J|)k}}{R^4}\sum_{\xi,\eta\in(2\pi\Z/R)^2\atop|\xi|,|\eta|\approx 2^k}
g_{I,J}(z,\xi,\eta).
\end{align*}
Now we can pass the $L^p$ norm inside the integral and the sum,
and then apply H\"older's inequality. From the bounds
$\|\mathcal K\|_{L^1}\lesssim 2^{-4k}$ and $\|\tilde g\|_{L^q}\lesssim 2^{4k+mk^+}\|a\|_{\mathcal L_m^q}$ follows the lemma with $f$ in place of $P_{[k-2,k+2]}f$.

When $k=-\log R+O(1)$, we can take $|I|=|J|=0$, in which case
$g_{I,J}(z,\xi,\eta)=|a(z,\frac{\xi+\eta}2)|\le|a|(z,\frac{\xi+\eta}2)$.
The same conclusion follows from the bound $\|\tilde g_{I,J}\|_{L^q}\lesssim R^{-4}\|a\|_{\mathcal L_m^q}$

To show the lemma itself, note in the definition (\ref{para-def}), $|\xi|/2<|\eta|<2|\xi|$, so $\varphi_k(\xi)>0$ implies $\varphi_{[k-2,k+2]}(\eta)=1$, and hence $P_kT_af=P_kT_aP_{[k-2,k+2]}f$.
\end{proof}

Paradifferential operators extract the ``quasilinear'' part of products,
leaving ``semilinear'' remainders.
\begin{definition}\label{H-def}
Given two functions $f$ and $g$, define
\[
H(f,g)=fg-T_fg-T_gf.
\]
\end{definition}

By Lemma \ref{para2diff} (iv) and (v), $P_kH(f,g)=P_kH(P_{>k-20}f,P_{>k-20}g)$. By Lemma \ref{Taf-Lp} and Lemma \ref{Lmq=Lq} we then get
\begin{lemma}\label{PkH-Lp}
For fixed $p,q,r\in[1,\infty]$ with $1/p=1/q+1/r$ we have
\[
\|P_kH(f,g)\|_{L^p}\lesssim \|P_{>k-20}f\|_{L^q}\|P_{>k-20}g\|_{L^r}.
\]
\end{lemma}

Next we show the commutator estimates of paradifferential operators.

\begin{definition}
Given symbols $a_1,\dots,a_n$, define the operator
\[
E(a_1,\dots,a_n)=T_{a_1}\cdots T_{a_n}-T_{a_1\cdots a_n}.
\]
\end{definition}

\begin{lemma}\label{Eaf-Lp}
For fixed $m_j\in\R$, $p,q_j,r\in[1,\infty]$ ($j=1,\dots,n$) with $1/p=1/q_1+\cdots+1/q_n+1/r$ we have
\begin{equation}\label{Eaf-Lp-bound}
\|P_kE(a_1,\dots,a_n)f\|_{L^p}
\lesssim 2^{(\sum_{j=1}^n m_j-1)k^+}\prod_{j=1}^n (\|a_j\|_{\mathcal L_{m_j}^{q_j}}+\|\nabla_x a_j\|_{\mathcal L_{m_j}^{q_j}}) \|P_{[k-2n,k+2n]}f\|_{L^r}.
\end{equation}
\end{lemma}

Roughly speaking, the operator $E(a_1,\dots,a_n)$ is one order smoother than either term on the right, so it can be thought of as an ``error term''.
\begin{proof}
Lemma \ref{prod-Lm} and Lemma \ref{Taf-Lp} allow us to use induction on $n$,
so it suffices to show the case when $n=2$.
If $k\le0$, the result also follows from the two lemmas because $2^{k^+}=1$.

Now we assume $k>0$, decompose $a_j=a_j^L+a_j^H$ ($j=1,2$), where
$a_j^L=P_{\le k-20}a_j$, and put
\begin{align*}
E^L(a_1,a_2)&=E(a_1^L,a_2^L),\\
E^H(a_1,a_2)&=E(a_1,a_2)-E^L(a_1,a_2)=E(a_1^H,a_2)+E(a_1^L,a_2^H).
\end{align*}

For $E^H(a_1,a_2)$, we use Lemma \ref{prod-Lm} and Lemma \ref{Taf-Lp} to get
\begin{align*}
\|P_kE(a_1^H,a_2)f\|_{L^p}
&\lesssim 2^{(m_1+m_2)k} \|a_1^H\|_{\mathcal L_{m_1}^{q_1}}\|a_2\|_{\mathcal L_{m_2}^{q_2}}\|P_{[k-4,k+4]}f\|_{L^r}\\
&\lesssim 2^{(m_1+m_2-1)k}(\|a_1\|_{\mathcal L_{m_1}^{q_1}}+\|\nabla_x a_1\|_{\mathcal L_{m_1}^{q_1}})\|a_2\|_{\mathcal L_{m_2}^{q_2}} \|P_{[k-4,k+4]}f\|_{L^r}.
\end{align*}
A similar bound, with $\nabla_x$ hitting $a_2$, holds for $P_kE(a_1^L,a_2^H)f$.

For $E^L(a_1,a_2)$, since $a_j^L=P_{\le k-20}a_j$, we have
\[
\mathcal F(P_kE(a_1^L,a_2^L)f)(\xi)
=\frac{C^2}{R^4}\sum_{\eta,\zeta\in(2\pi\Z/R)^2} \hat A(\xi,\eta,\zeta)\hat f(\zeta),
\]
where
\begin{align*}
\hat A&=\mathcal F_xa_1^L\left( \xi-\eta,\frac{\xi+\eta}2 \right)
\mathcal F_xa_2^L\left( \eta-\zeta,\frac{\eta+\zeta}2 \right)-
\mathcal F_xa_1^L\left( \xi-\eta,\frac{\xi+\zeta}2 \right)
\mathcal F_xa_2^L\left( \eta-\zeta,\frac{\xi+\zeta}2 \right)\\
&=\hat A_1(\xi,\eta,\zeta)-\hat A_2(\xi,\eta,\zeta),\\
\hat A_1&=\int_0^1 (\nabla_\zeta\mathcal F_xa_1^L)\left( \xi-\eta,\frac{\xi+\zeta+t(\eta-\zeta)}2 \right)\cdot\frac{\eta-\zeta}2\mathcal F_xa_2^L\left( \eta-\zeta,\frac{\xi+\zeta+t(\eta-\xi)}2 \right)dt\\
&=\frac1{2i}\int_0^1 \mathcal F_x(\nabla_\zeta a_1^L)\left( \xi-\eta,\frac{\xi+\zeta+t(\eta-\zeta)}2 \right)\cdot\mathcal F_x(\nabla_xa_2^L)\left( \eta-\zeta,\frac{\xi+\zeta+t(\eta-\xi)}2 \right)dt,\\
\hat A_2&=\int_0^1\mathcal F_xa_1^L\left( \xi-\eta,\frac{\xi+\zeta+t(\eta-\zeta)}2 \right)\frac{\xi-\eta}2\cdot(\nabla_\zeta\mathcal F_xa_2^L)\left( \eta-\zeta,\frac{\xi+\zeta+t(\eta-\xi)}2 \right)dt\\
&=\frac1{2i}\int_0^1 \mathcal F_x(\nabla_x a_1^L)\left( \xi-\eta,\frac{\xi+\zeta+t(\eta-\zeta)}2 \right)\cdot\mathcal F_x(\nabla_\zeta a_2^L)\left( \eta-\zeta,\frac{\xi+\zeta+t(\eta-\xi)}2 \right)dt.
\end{align*}
Computing in the same way as (\ref{diff-quot}), the Schwartz kernel for $P_kE^L(a_1,a_2)$ is then
\begin{align*}
S(x,y)
&=\iint \frac{C''}{R^6}\sum_{\xi,\eta,\zeta\in(2\pi\Z/R)^2}
e^{i(\xi\cdot(x-z)+\eta\cdot(z-z')+\zeta\cdot(z'-y))}\\
&\times\frac{\left( \prod_{j=1}^{|I|} \Delta_{1/R}^{\xi_{I_j}} \right)
\left( \prod_{j=1}^{|J|} \Delta_{1/R}^{\eta_{J_j}} \right)
\left( \prod_{j=1}^{|K|} \Delta_{1/R}^{\zeta_{K_j}} \right)
\left( A(z,z',\xi,\eta,\zeta)\varphi_k(\xi) \right)dzdz'}
{R^{|I|+|J|+|K|}\prod_{j=1}^{|I|}(e^{2\pi i(x_{I_j}-z_{I_j})/R}-1)
\prod_{j=1}^{|J|}(e^{2\pi i(z_{J_j}-z'_{J_j})/R}-1)
\prod_{j=1}^{|K|}(e^{2\pi i(z'_{K_j}-y_{K_j})/R}-1)},
\end{align*}
where $A=A_1+A_2$,
\begin{align*}
A_1&=\frac1{2i}\int_0^1 (\nabla_\zeta a_1^L)\left( z,\frac{\xi+\zeta+t(\eta-\zeta)}2 \right)\cdot(\nabla_xa_2^L)\left( z',\frac{\xi+\zeta+t(\eta-\xi)}2 \right)dt,\\
A_2&=\frac1{2i}\int_0^1 (\nabla_xa_1^L)\left( z,\frac{\xi+\zeta+t(\eta-\zeta)}2 \right)\cdot(\nabla_\zeta a_2^L)\left( z',\frac{\xi+\zeta+t(\eta-\xi)}2 \right)dt.
\end{align*}

Note that this expression is similar to (\ref{diff-quot}),
with $\nabla_\zeta a_i^L\nabla_xa_j^L$ ($\{i,j\}=\{1,2\}$) replacing $a$.
Then the same argument as in the proof of Lemma \ref{Taf-Lp} shows that
\[
|P_kE(a_1^L,a_2^L)f(x)|\lesssim \min_{|I|+|J|+|K|\le7} \int \mathcal K(y,z,z')f(x-y)\tilde g(x-z,x-z')dydzdz',
\]
where
\begin{align*}
\mathcal K(y,z,z')&=\frac{1}{(1+2^k(\|y\|+\|z\|+\|z'\|))^7},\\
|\tilde g(z,z')|&\lesssim 2^{6k}\max_{l+l'\le7\atop \{i,j\}=\{1,2\}}
\sup_{|\zeta|,|\zeta'|\approx 2^k} |\zeta|^l|\nabla_\zeta^{l+1}a_i^L(z,\zeta)||\zeta'|^{l'}|\nabla_\zeta^l\nabla_xa_j^L(z',\zeta')|.
\end{align*}
Now we can pass the $L^p$ norm inside the integral and the sum,
and then apply H\"older's inequality. We have
$\|\mathcal K\|_{L^1}\lesssim 2^{-6k}$ and $\|\tilde g(x-z,x-z')\|_{L_x^{q_1q_2/(q_1+q_2)}}\lesssim 2^{(5+m_1+m_2)k}C$,
where $C$ denotes the product over $j$ on the right-hand side of (\ref{Eaf-Lp-bound}). Then it follows that
\[
\|P_kE^L(a_1,a_2)f\|_{L^p}\lesssim 2^{(m_1+m_2-1)k}C\|P_{[k-4,k+4]}f\|_{L^r},
\]
where we have replaced $f$ with $P_{[k-4,k+4]}f$ as before.
Combining this with the bound for $P_kE^H(a_1,a_2)f$ we get the lemma.
\end{proof}

\subsection{Multilinear paraproduct estimates}
We also need to bound multilinear paraproducts on the torus.
We will only use multipliers that are restrictions of Schwartz functions to the lattice in the frequency space.
\begin{definition}
If $m$ is a Schwartz function on $(\R^2)^n$, define
\begin{align*}
\|m\|_{S^\infty}&=\|\mathcal Fm\|_{L^1},\\
\|m\|_{S^\infty_{k_1,\cdots,k_n;k}}&=\|\varphi_k(\xi_1+\cdots+\xi_n)m(\xi_1,\cdots,\xi_n)\varphi_{k_1}(\xi_1)\cdots\varphi_{k_n}(\xi_n)\|_{S^\infty}.
\end{align*}
\end{definition}
\begin{definition}
Throughout the paper we let $K$ be the largest of $k$, $k_1,\dots,k_n$.
\end{definition}

The next lemma allows us to estimate the $S^\infty$ norm of various symbols.
\begin{lemma}\label{Soo-Cn}
(i) $\|m_1m_2\|_{S^\infty}\le\|m_1\|_{S^\infty}\|m_2\|_{S^\infty}$.

(ii) For $k_1,\dots,k_n,k\in\Z$ we have
\[
\|m\cdot\otimes_{j=1}^n\varphi_{k_j}\|_{S^\infty}
\lesssim_n \sum_{l=0}^{n+1} \sum_{j=0}^n 2^{lk_j}\|
\varphi_{[k_j-1,k_j+1]}\nabla_j^lm\|_{L^\infty}.
\]

(iii)
\[
\|m\|_{S^\infty((\R/\Z)^{2n})}\le\|m\|_{S^\infty(\R^{2n})},
\]
that is, we can bound the $S^\infty$ norm of a multiplier restricted to the frequency lattice $(2\pi\Z)^{2n}$ by the $S^\infty$ norm of the multiplier itself. By scaling symmetry, this also applies to multipliers restricted to a rescaled lattice.
\end{lemma}
\begin{proof}
(i) follows directly from the definition. (ii) is \cite{DeIoPa}, Lemma 3.3.
(iii) follows from the definition, the Poisson summation formula
\[
\mathcal F_{(2\pi\Z)^{2n}\to(\R/\Z)^{2n}}(m)(x)
=\sum_{z\in\Z^{2n}}\mathcal F_{\R^{2n}\to\R^{2n}}m(x+z)
\]
and the triangle inequality.
\end{proof}

The $L^p$ boundedness of a paraproduct of functions is well known.
\begin{lemma}\label{paraprod}
Fix $p,p_j\in[1,\infty]$ ($j=1,\dots,n$) and $1/p=1/p_1+\cdots+1/p_n$. Let
\[
\mathcal Ff(\xi)
=\frac{1}{R^{2n-2}}\sum_{\xi_j\in(2\pi\Z/R)^2\atop\xi_1+\dots+\xi_n=\xi} m(\xi_1,\dots,\xi_n)\prod_{j=1}^n \mathcal F f_j(\xi_j).
\]
Then
\[
\|f\|_{L^p}\lesssim_n \|m\|_{S^\infty}\prod_{j=1}^n\|f_j\|_{L^{p_j}},
\]
where the $L^p$ norms are taken on $(\R/R\Z)^2$.
\end{lemma}
\begin{proof}
The Schwartz kernel of the operator is $\mathcal Fm(x-x_1,\dots,x-x_n)$.
\end{proof}

\section{Quartic energy estimates}\label{EneEst}
In this section we will obtain a quartic energy estimate of the form
\[
\mathcal E(t)=\mathcal E(0)+\int_0^t \mathcal E(s)\|U(s)\|_X^2ds.
\]

\subsection{Defining the quartic energy}
\begin{definition}
For an integer $N\ge2$ define
\begin{align*}
\mathcal U&=X+iT_{\sqrt{1+\rho}}Y, & \mathcal E&=\|P_{\ge0}\mathcal U\|_{H^N}^2=\|P_{\ge0}\Lambda^N\mathcal U\|_{L^2}^2.
\end{align*}
\end{definition}

We first show that $\mathcal E$ is close to the usual $H^N$ norm, up to a cubic error.
\begin{proposition}\label{E-HN}
If $\|U\|_{H^2}$ is sufficiently small then
\[
|\mathcal E-\|P_{\ge0}U\|_{H^N}^2|\lesssim \|U\|_{H^N}^3.
\]
\end{proposition}
\begin{proof}
From Lemma \ref{Taf-Lp} and Sobolev embedding it follows that
\[
\|P_{\ge0}(\mathcal U-U)\|_{H^N}
=\|P_{\ge0}T_{\sqrt{1+\rho}-1}Y\|_{H^N}\lesssim \|\rho\|_{L^\infty}\|U\|_{H^N}\lesssim \|U\|_{H^N}^2.
\]
If $\|U\|_{H^2}$ is sufficiently small, the above also gives $\|P_{\ge0}\mathcal U\|_{H^N}\lesssim \|U\|_{H^N}$. Combining the two bounds shows the claim.
\end{proof}

The rest of this section is devoted to estimating $d\mathcal E/dt$.
To begin with, the evolution equation for $\mathcal U$ is
\begin{align*}
\partial_t\mathcal U&=\Lambda(Y-iT_{\sqrt{1+\rho}}X)-\Lambda R_j(T_\rho v_j+T_{v_j}\rho+H(\rho,v_j))\\
&-iT_{\sqrt{1+\rho}}|\nabla|(T_{v_j}v_j+H(v_j,v_j)/2)+iT_{\partial_t\sqrt{1+\rho}}Y.
\end{align*}
Using the definition of $E$, Lemma \ref{para2diff} (ii) and $\zeta_j\zeta_j=|\zeta|^2$ we get
\begin{align*}
(\partial_t+iT_{\sqrt{1+\rho}\Lambda(\zeta)})\mathcal U
&=-iE(\Lambda(\zeta),T_{\sqrt{1+\rho}})X-E(\sqrt{1+\rho}\Lambda(\zeta),\sqrt{1+\rho})Y\\
&+E(\Lambda(\zeta)\zeta_j/|\zeta|,\rho,\zeta_j/|\zeta|)Y-\Lambda R_jT_{v_j}\rho-iT_{\sqrt{1+\rho}}|\nabla|T_{v_j}v_j\\
&-\Lambda R_jH(\rho,v_j)-iT_{\sqrt{1+\rho}}|\nabla|H(v_j,v_j)/2
+iT_{\partial_t\sqrt{1+\rho}}Y.
\end{align*}
Using $E(a,1)=E(1,b)=0$, $[T_a,T_b]=E(a,b)-E(b,a)$ and the bilinearity of $E$ we get
\begin{equation}\label{cal-Ut}
(\partial_t+iT_{\sqrt{1+\rho}\Lambda(\zeta)}+iT_{v\cdot\zeta})\mathcal U=\mathcal{Q+S+C},
\end{equation}
where
\begin{align*}
\mathcal Q&=-iE(\Lambda(\zeta),T_{\sqrt{1+\rho}})X-E(\Lambda(\zeta),\sqrt{1+\rho}-1)Y+E(\Lambda(\zeta)\zeta_j/|\zeta|,\rho,\zeta_j/|\zeta|)Y\\
&-iE(\Lambda(\zeta)\zeta_j/|\zeta|,v_j,|\zeta|/\Lambda(\zeta))X+E(|\zeta|,v_j,\zeta_j/|\zeta|)Y+iT_{|\nabla|Y/2}Y,\\
\mathcal S&=-\Lambda R_jH(\rho,v_j)-i|\nabla|H(v_j,v_j)/2,\\
\mathcal C&=-E((\sqrt{1+\rho}-1)\Lambda(\zeta),\sqrt{1+\rho}-1)Y+T_{\sqrt{1+\rho}-1}E(|\zeta|,v_j,\zeta_j/|\zeta|)Y\\
&+[iT_{v\cdot\zeta},T_{\sqrt{1+\rho}-1}]Y-iT_{\sqrt{1+\rho}-1}|\nabla|H(v_j,v_j)/2+iT_{\partial_t\sqrt{1+\rho}-|\nabla|Y/2}Y
\end{align*}
are quasilinear quadratic, semilinear quadratic and cubic terms, respectively.

By Lemma \ref{para2diff} (i), $T_{\sqrt{1+\rho}\Lambda(\zeta)+v\cdot\zeta}$ is self-adjoint,
so $\langle T_{\sqrt{1+\rho}\Lambda(\zeta)+v\cdot\zeta}f,f \rangle\in\R$,
where the inner product is taken on the torus.
Now we can decompose $d\mathcal E/dt$ accordingly:
\[
\frac{d}{dt}\mathcal E
=2\Re\langle (\partial_t+iT_{\sqrt{1+\rho}\Lambda(\zeta)}+iT_{v\cdot\zeta})P_{\ge0}\Lambda^N\mathcal U,
P_{\ge0}\Lambda^N\mathcal U \rangle
=2(\mathcal E_Q+\mathcal E_S+\mathcal E_4),
\]
where (note that $[T_{\Lambda(\zeta)},P_{\ge0}\Lambda^N]=[\Lambda,P_{\ge0}\Lambda^N]=0$)
\begin{align*}
\mathcal E_Q&=\Re\langle[iT_{(\sqrt{1+\rho}-1)\Lambda(\zeta)+v\cdot\zeta},P_{\ge0}\Lambda^N]\mathcal U
+P_{\ge0}\Lambda^N\mathcal Q,P_{\ge0}\Lambda^N\mathcal U\rangle,\\
\mathcal E_S&=\Re\langle P_{\ge0}\Lambda^N\mathcal S,P_{\ge0}\Lambda^NU \rangle,\\
\mathcal E_4&=\Re\langle P_{\ge0}\Lambda^N\mathcal S,iP_{\ge0}\Lambda^NT_{\sqrt{1+\rho}-1}Y \rangle+\Re\langle P_{\ge0}\Lambda^N\mathcal C,P_{\ge0}\Lambda^N\mathcal U\rangle
\end{align*}
are quasilinear cubic, semilinear cubic and quartic energies, respectively.

\subsection{Bounding the quartic energy}
\begin{proposition}\label{E4}
If $\|U\|_{H^2}$ is sufficiently small then
\[
|\mathcal E_4|\lesssim \|U\|_X^2\|U\|_{H^N}^2.
\]
\end{proposition}
\begin{proof}
By Lemma \ref{PkH-Lp}, Lemma \ref{Taf-Lp}, Lemma \ref{Lmq=Lq} and Lemma \ref{XZ-bound} (iii), for $k\ge-2$ we have
\begin{align}
\label{PkS}
\|P_k\mathcal S\|_{H^N}&\lesssim 2^{(N+1)k}\|P_{>k-20}(\rho,v)\|_{L^\infty}\|P_{>k-20}v\|_{L^2}\lesssim 2^{Nk}\|U\|_X\|P_{>k-20}U\|_{L^2},\\
\label{TrY}
\|P_kT_{\sqrt{1+\rho}-1}Y\|_{H^N}&\lesssim \|\rho\|_{L^\infty}\|P_{[k-2,k+2]}U\|_{H^N}\lesssim \|U\|_X\|P_{[k-2,k+2]}U\|_{H^N}.
\end{align}
The desired bound for the first term follows after taking the $\ell^2$ sum in $k$.

For the second term, by Proposition \ref{E-HN} and Lemma \ref{XZ-bound} (iii) it suffices to show
\begin{equation}\label{P>0C}
\|P_{\ge0}\Lambda^N\mathcal C\|_{H^N}\lesssim \|(\rho,v)\|_{W^{1,\infty}}^2\|U\|_{H^N}.
\end{equation}
The desired bounds for the first five terms of $P_{\ge0}\mathcal C$ follow from Lemma \ref{Sm-Lm}, Lemma \ref{Lmq=Lq}, Lemma \ref{prod-Lm}, Lemma \ref{Taf-Lp} and Lemma \ref{Eaf-Lp}. To bound the sixth term we also need to bound $\|P_k\Im\mathcal S\|_{H^N}$, which can also be bounded using (\ref{PkS}). To bound the last term we also need the identity
\[
\partial_t\sqrt{1+\rho}=\frac{-\nabla\cdot((1+\rho)v)}{2\sqrt{1+\rho}}
=\frac{|\nabla|Y-(\sqrt{1+\rho}-1)\nabla\cdot v}2-\frac{v\cdot\nabla\rho}{2\sqrt{1+\rho}}.
\]
\end{proof}

\subsection{Bounding the semilinear energy}
\begin{proposition}\label{Es}
If $N\ge9$ then
\[
\left| \int_0^t \mathcal E_S(s)ds \right|\lesssim \|U\|_{L^\infty([0,t])H^N}^3+\|U\|_{L^2([0,t])X}^2\|U\|_{L^\infty([0,t])H^N}^2.
\]
\end{proposition}
\begin{proof}
Let $U_+=U$ and $U_-=\bar U$. Then $\mathcal E_S$ is a linear combination of terms of the form $\Re \mathcal E_S^{\mu\nu}$, where
\[
\mathcal E_S^{\mu\nu}
=\langle P_{\ge0}\Lambda^{N+1}T_3H(T_1U_\mu,T_2U_\nu),P_{\ge0}\Lambda^NU \rangle,
\]
$T_1$, $T_2$ and $T_3$ are Calderon--Zygmund operators,
and $\mu,\nu\in\{+,-\}$. Let
\begin{align}
\label{F-def}
\Phi_{\mu\nu}(\xi_1,\xi_2)
&=\Lambda(\xi_1+\xi_2)-\mu\Lambda(\xi_1)-\nu\Lambda(\xi_2),\\
\nonumber
I_S^{\mu\nu}[f_1,f_2,f_3]
&=\frac{C^2}{R^4}\sum_{\xi_j\in(2\pi\Z/R)^2}
\frac{1-\varphi_{\le-10}\left( \frac{|\xi_1|}{|\xi_1+2\xi_2|} \right)-\varphi_{\le-10}\left( \frac{|\xi_2|}{|\xi_2+2\xi_1|} \right)}
{\Phi_{\mu\nu}(\xi_1,\xi_2)}\hat f_1(\xi_1)\hat f_2(\xi_2)\overline{\hat f_3(\xi_1+\xi_2)},\\
\nonumber
I_S^{\mu\nu}&=I_S^{\mu\nu}[T_1U_\mu,T_2U_\nu,P_{\ge0}^2\Lambda^{2N+1}T_3^*U].
\end{align}

By (\ref{XtYt}), the evolution equation for $U$ is
\begin{align}\label{Ut}
U_t&=-i\Lambda U-N, & N&=\Lambda R_j(\rho v_j)+i|\nabla|(v^2/2).
\end{align}
Let $N_+=N$ and $N_-=\bar N$. Then
\begin{align*}
\frac{dI_S^{\mu\nu}}{dt}&=I_S^{\mu\nu}[T_1(U_\mu)_t,T_2U_\nu,P_{\ge0}^2\Lambda^{2N+1}T_3^*U]
+I_S^{\mu\nu}[T_1U_\mu,T_2(U_\nu)_t,P_{\ge0}^2\Lambda^{2N+1}T_3^*U]\\
&+I_S^{\mu\nu}[T_1U_\mu,T_2U_\nu,P_{\ge0}^2\Lambda^{2N+1}T_3^*U_t]\\
&=\mathcal E_S^{\mu\nu}
-I_S^{\mu\nu}[T_1N_\mu,T_2U_\nu,P_{\ge0}^2\Lambda^{2N+1}T_3^*U]
-I_S^{\mu\nu}[T_1U_\mu,T_2N_\nu,P_{\ge0}^2\Lambda^{2N+1}T_3^*U]\\
&-I_S^{\mu\nu}[T_1U_\mu,T_2U_\nu,P_{\ge0}^2\Lambda^{2N+1}T_3^*N].
\end{align*}
Integration by parts in time gives
\begin{align}
\label{Es-IBP1}
&\int_0^t \mathcal E_S^{\mu\nu}(s)ds=I_S^{\mu\nu}(t)-I_S^{\mu\nu}(0)\\
\label{Es-IBP2}
+&\int_0^t (I_S^{\mu\nu}[T_1N_\mu,T_2U_\nu,P_{\ge0}^2\Lambda^{2N+1}T_3^*U]
+I_S^{\mu\nu}[T_1U_\mu,T_2N_\nu,P_{\ge0}^2\Lambda^{2N+1}T_3^*U])(s)ds\\
\label{Es-IBP3}
+&\int_0^t I_S^{\mu\nu}[T_1U_\mu,T_2U_\nu,P_{\ge0}^2\Lambda^{2N+1}T_3^*N](s)ds.
\end{align}

To bound $I_S^{\mu\nu}$, we need a bound of the $S^\infty$ norm of $\Phi_{\mu\nu}^{-1}$.
\begin{lemma}\label{1/F-Soo}
For $k_1$, $k_2$, $k_3\in\Z$ we have
\[
\|\Phi_{\mu\nu}^{-1}\|_{S^\infty_{k_1,k_2;k_3}}\lesssim 2^{7\min(k_1,k_2,k_3)^+}.
\]
\end{lemma}
\begin{proof}
From \cite{IoPa1}, Lemma 5.1 it follows that
\begin{align}\label{1/F-C0}
|\Phi_{\mu\nu}|&\gtrsim (1+\min(|\xi_1|,|\xi_2|,|\xi_1+\xi_2|))^{-1}, &
|\nabla\Phi_{\mu\nu}|&\lesssim |\Phi_{\mu\nu}|
\end{align}
Then for $L\ge1$ we have $|\nabla^L\Phi_{\mu\nu}|\lesssim_L |\Phi_{\mu\nu}|$.
By Leibniz's rule we have
\begin{equation}\label{1/F-CL}
|\nabla^L(\Phi_{\mu\nu}^{-1})|\lesssim_L |\Phi_{\mu\nu}|^{-1}\lesssim 1+\min(|\xi_1|,|\xi_2|,|\xi_1+\xi_2|).
\end{equation}

Without loss of generality we assume $k_1\le k_2\le k_3$. We distinguish two cases.

{\bf Case 1:} $k_1\ge k_2-2$. Then the bound follows from (\ref{1/F-CL}) and Lemma \ref{Soo-Cn} (ii), with $n=2$.

{\bf Case 2:} $k_1\le k_2-3$. We still have $|\nabla^L_{\xi_1}\Phi_{\mu\nu}|\lesssim_L |\Phi_{\mu\nu}|$. For $\nabla^L_{\xi_2}\Phi_{\mu\nu}$ we further distinguish two cases.

{\bf Case 2.1:} $\nu=+$. By the fundamental theorem of calculus,
\[
|\nabla^L_{\xi_2}\Phi_{\mu+}|
=|\nabla^L\Lambda(\xi_1+\xi_2)-\nabla^L\Lambda(\xi_2)|
\le\int_0^1 |\xi_1\cdot\nabla\nabla^L\Lambda(t\xi_1+\xi_2)|dt
\lesssim_L |\xi_1|(1+|\xi_2|)^{-L},
\]
so $|\xi_2|^L|\nabla^L_{\xi_2}\Phi_{\mu+}|\lesssim_L |\xi_1|$.
By (\ref{1/F-C0}) and Leibniz's rule,
$|\xi_2|^L|\nabla^L_{\xi_2}(\Phi_{\mu+}^{-1})|\lesssim_L |\xi_1|^{2L+1}$,
and the bound follows Lemma \ref{Soo-Cn} (ii).

{\bf Case 2.2:} $\nu=-$. Then for $L\ge1$ we have $|\xi_2|^L|\nabla_{\xi_2}^L\Phi_{\mu-}|\lesssim_L |\xi_2|^L(1+|\xi_2|)^{-L+1}\le|\xi_2|\lesssim\Phi_{\mu-}$, so by (\ref{1/F-C0}) and Lemma \ref{Soo-Cn} (ii), the bound can actually be improved to $\|\Phi_{\mu-}^{-1}\|_{S^\infty_{k_1,k_2;k_3}}\lesssim 2^{-k_3^+}$.
\end{proof}

Now we bound (\ref{Es-IBP1}), (\ref{Es-IBP2}) and (\ref{Es-IBP3}).
By Lemma \ref{paraprod} and Lemma \ref{1/F-Soo},
\[
|I_S^{\mu\nu}[T_1P_{k_1}f_1,T_2P_{k_2}f_2,P_{\ge0}^2\Lambda^{2N+1}T_3^*P_{k_3}f_3]|
\lesssim 2^{(2N+8)k_3^+}\|P_{k_1}f_1\|_{L^\infty}\|P_{k_2}f_2\|_{L^2}\|P_{k_3}f_3\|_{L^2}.
\]
Thanks to the factors $P_{\ge0}$ and $\varphi_{\le-10}$,
this term vanishes unless $k_3\ge-1$ and $k_1, k_2\ge k_3-20$,
in which case, using Bernstein's inequality, it can be bounded by
\begin{equation}
2^{-N|k_2-k_3|}\|P_{k_1}f_1\|_{W^{8+m,\infty}}\|P_{k_2}f_2\|_{H^N}\|P_{k_3}f_3\|_{H^{N-m}}. \tag{$m\ge0$}
\end{equation}
A similar bound with $f_1$ and $f_2$ swapped holds.
The additive restriction of frequencies implies $|k_1-k_2|=O(1)$,
so summing over $k_1$, $k_2$ and $k_3$ using the Cauchy--Schwarz inequality gives
\begin{align}
\nonumber
|I_S^{\mu\nu}[T_1f_1,T_2f_2,P_{\ge0}^2\Lambda^{2N+1}T_3^*f_3]|
&\lesssim \sup_{k_1\in\Z} \|P_{k_1}f_1\|_{W^{8+m,\infty}} \sum_{k,l\in\Z} 2^{-N|l|}\|P_{k+l}f_2\|_{H^N}\|P_kf_3\|_{H^{N-m}}\\
&\lesssim \sup_{k\in\Z} \|P_kf_1\|_{W^{8+m,\infty}} \|f_2\|_{H^N}\|f_3\|_{H^{N-m}}.
\label{Is}
\end{align}

By (\ref{Is}) with $f_1=U_\mu$, $f_2=U_\nu$, $f_3=U$, $m=0$ and $N\ge9$,
\[
|(\ref{Es-IBP1})|\lesssim \|U(t)\|_{H^N}^3+\|U(0)\|_{H^N}^3
\lesssim \|U\|_{L^\infty([0,t])H^N}^3.
\]
By (\ref{Ut}) and Lemma \ref{XZ-bound} (iii), $\|P_kN\|_{W^{8,\infty}}\lesssim \|(\rho,v)\|_{W^{9,\infty}}^2\lesssim \|U\|_X^2$, so by (\ref{Is}) with $f_1=N_\mu$, $f_2=U_\nu$, $f_3=U$, $m=0$
(and its symmtric version),
\[
|(\ref{Es-IBP2})|\lesssim \|U\|_{L^2([0,t])X}^2\|U\|_{L^\infty([0,t])H^N}^2.
\]
By (\ref{Ut}) and the Sobolev multiplication theorem, $\|N\|_{H^{N-1}}\lesssim \|U\|_X\|U\|_{H^N}$, so by (\ref{Is}) with $f_1=U_\mu$, $f_2=U_\nu$, $f_3=N$, $m=1$, the same holds for (\ref{Es-IBP3}).

Combining the three bounds shows the proposition.
\end{proof}

\subsection{Bounding the quasilinear energy}
\begin{proposition}\label{Eq}
If $N\ge10$ and $\|U\|_{H^3}$ is sufficiently small then
\[
\left| \int_0^t \mathcal E_Q(s)ds \right|\lesssim \|U\|_{L^\infty([0,t])H^N}^3+\|U\|_{L^2([0,t])X}^2\|U\|_{L^\infty([0,t])H^N}^2.
\]
\end{proposition}
\begin{proof}
Up to a quartic error like the second term in $\mathcal E_4$,
which can be bounded using (\ref{TrY}), Lemma \ref{Lmq=Lq}, Lemma \ref{prod-Lm}, Lemma \ref{Taf-Lp}, Lemma \ref{Eaf-Lp} and the bound
\[
\|\sqrt{1+\rho}-1-\rho/2\|_{W^{1,\infty}}\lesssim \|\rho\|_{W^{1,\infty}}^2,
\]
$\mathcal E_Q$ is a sum of the terms $\mathcal E_Q^{\mu\nu}$, where
\begin{align*}
\mathcal E_Q^{\mu\nu}&=\Re\frac{C^2}{R^4} \sum_{\xi_j\in(2\pi\Z/R)^2}
q(\xi_1,\xi_2)\hat U_\mu(\xi_1)\mathcal{\hat U}_\nu(\xi_2)\overline{\mathcal{\hat U}(\xi_1+\xi_2)},\\
q(\xi_1,\xi_2)&=n_1(\xi_1)n_2(\xi_2)n_3(\xi_1+\xi_2)
\left[ n_4(\xi_1+\xi_2)n_5(\xi_2)-(n_4n_5)\left(\frac{\xi_1+2\xi_2}2 \right) \right]\\
&\times\varphi_{\le-10}\left( \frac{|\xi_1|}{|\xi_1+2\xi_2|} \right),
\end{align*}
$n_j\in S^{m_j}_{1,0}$ ($1\le j\le5$), $\sum_{j=1}^5 m_j=2N+1$ and
$\supp n_2\cup\supp n_3\subset\supp\varphi_{\ge0}$. Let
\begin{align*}
I_Q^{\mu\nu}[f_1,f_2,f_3]
&=\Re\frac{C^2}{R^4} \sum_{\xi_j\in(2\pi\Z/R)^2}
\frac{q(\xi_1,\xi_2)}{\Phi_{\mu\nu}(\xi_1,\xi_2)}\hat f_1(\xi_1)\hat f_2(\xi_2)\overline{\hat f_3(\xi_1+\xi_2)},\\
I_Q^{\mu\nu}&=I_Q^{\mu\nu}[U_\mu,\mathcal U_\nu,\mathcal U].
\end{align*}
Similarly integration by parts in time gives
\begin{align}
\label{Eq-IBP1}
\int_0^t \mathcal E_Q^{\mu\nu}(s)ds&=I_Q^{\mu\nu}(t)-I_Q^{\mu\nu}(0)\\
\label{Eq-IBP2}
&+\int_0^t I_Q^{\mu\nu}[N_\mu,\mathcal U_\nu,\mathcal U](s)ds\\
\label{Eq-IBP3}
&+\int_0^t (I_Q^{\mu\nu}[U_\mu,(\mathcal U_\nu)_t+i\Lambda\mathcal U_\nu,\mathcal U]+I_Q^{\mu\nu}[U_\mu,\mathcal U_\nu,\mathcal U_t+i\Lambda\mathcal U])(s)ds.
\end{align}
The bound then follows from the corresponding bounds for (\ref{Eq-IBP1}), (\ref{Eq-IBP2}) and (\ref{Eq-IBP3}), to be shown in Proposition \ref{Eq-IBP12} and Proposition \ref{Eq-IBP33} below. 
\end{proof}

To estimate $I_Q^{\mu\nu}$, we need to bound the $S^\infty$ norm of the $q$ multiplier.
\begin{lemma}\label{q-Soo}
For $k_1$, $k_2$, $k_3\in\Z$ we have
\begin{align*}
\|q\|_{S^\infty_{k_1,k_2;k_3}}&\lesssim 2^{2Nk_3^++k_1}1_{k_1\le k_3-6,|k_2-k_3|\le1},\\
\|\nabla_{\xi_2}q\|_{S^\infty_{k_1,k_2;k_3}}&\lesssim 2^{(2N-1)k_3^++k_1}1_{k_1\le k_3-6,|k_2-k_3|\le1}.
\end{align*}
\end{lemma}
\begin{proof}
The support part comes from the $\varphi_{\le-10}$ factor.
The bound follows from the identity
\[
n_4(\xi_1+\xi_2)n_5(\xi_2)-(n_4n_5)\left(\frac{\xi_1+2\xi_2}2 \right)\\
=\frac12\int_0^1 \xi_1\cdot\nabla(n_4(\xi_t)n_5(\eta_t))dt,
\]
where $\xi_t=((1+t)\xi_1+2\xi_2)/2$ and $\eta_t=((1-t)\xi_1+2\xi_2)/2$.
Then we use Lemma \ref{Soo-Cn} (i) and (ii) to bound the $S^\infty$ norm of the integrand. The bound on $\nabla_{\xi_2}q$ follows in a similar way.
\end{proof}

\begin{proposition}\label{Eq-IBP12}
If $N\ge10$ and $\|U\|_{H^2}$ is sufficiently small then
\begin{align*}
|(\ref{Eq-IBP1})|&\lesssim \|U\|_{L^\infty([0,t])H^N}^3, &
|(\ref{Eq-IBP2})|&\lesssim \|U\|_{L^2([0,t])X}^2\|U\|_{L^\infty([0,t])H^N}^2.
\end{align*}
\end{proposition}
\begin{proof}
By Lemma \ref{1/F-Soo} and Lemma \ref{q-Soo}, $\|q/\Phi_{\mu\nu}\|_{S^\infty_{k_1,k_2;k_3}}\lesssim 2^{2Nk_3^++7k_1^++k_1}1_{|k_2-k_3|\le1}$,
so by Lemma \ref{paraprod},
\[
|I_Q^{\mu\nu}[P_{k_1}f_1,P_{k_2}f_2,P_{k_3}f_3]|
\lesssim 2^{2Nk_3^++7k_1^++k_1}1_{|k_2-k_3|\le1}
\|P_{k_1}f_1\|_{L^\infty}\|P_{k_2}f_2\|_{L^2}\|P_{k_3}f_3\|_{L^2}.
\]
Summing over $k_1\in\Z$ and $|k_2-k_3|\le1$ using the Cauchy--Schwarz inequality gives
\begin{equation}\label{Iq}
|I_Q^{\mu\nu}[f_1,f_2,f_3]|\lesssim \sum_{k\in\Z} 2^{7k^++k}\|P_kf_1\|_{L^\infty}\|f_2\|_{H^N}\|f_3\|_{H^N}.
\end{equation}
Put $f_1=U_\mu$, $f_2=\mathcal U_\nu$ and $f_3=\mathcal U$ in (\ref{Iq}).
By Lemma \ref{XZ-bound} (ii), Proposition \ref{E-HN} and $N\ge10$,
\[
|(\ref{Eq-IBP1})|\lesssim \|U(t)\|_{H^N}^3+\|U(0)\|_{H^N}^3
\lesssim \|U\|_{L^\infty([0,t])H^N}^3.
\]
Put $f_1=N_{\mu}$, $f_2=\mathcal U_\nu$, $f_3=\mathcal U$ in (\ref{Iq}).
The integrand in (\ref{Eq-IBP2}) is bounded by
\[
|I_Q^{\mu\nu}[N_\mu,\mathcal U_\nu,\mathcal U]|
\lesssim \sum_{k\in\Z} 2^{8k^++k}\|P_k(\rho v,v^2)\|_{L^\infty}\|\mathcal U\|_{H^N}^2
\lesssim \|(\rho v,v^2)\|_{B^9_{\infty,1}}\|U\|_{H^N}^2.
\]
Since $M\ge9$, by [Tr], Theorem 2 (i) and Lemma \ref{XZ-bound} (iii) we have
\[
\|(\rho v,v^2)\|_{B^9_{\infty,1}}
\lesssim \|(\rho,v)\|_{B^9_{\infty,1}}^2\lesssim \|U\|_X^2.
\]
Integrating in time gives
\[
|(\ref{Eq-IBP2})|\lesssim \|U\|_{L^2([0,t])X}^2\|U\|_{L^\infty([0,t])H^N}^2.
\]
\end{proof}

\begin{proposition}\label{Eq-IBP33}
If $\|U\|_{H^3}$ is sufficiently small then
\[
|(\ref{Eq-IBP3})|\lesssim \|U\|_{L^2([0,t])X}^2\|U\|_{L^\infty([0,t])H^N}^2.
\]
\end{proposition}
\begin{proof}
By (\ref{cal-Ut}), the integrand of (\ref{Eq-IBP3}) becomes
\begin{align}
\label{Iq1}
&I_Q^{\mu\nu}[U_\mu,(\mathcal{Q+S+C})_\nu,\mathcal U]+I_Q^{\mu\nu}[U_\mu,U_\nu,\mathcal{Q+S+C}]\\
\label{Iq2}
-&I_Q^{\mu\nu}[U_\mu,(iT_{v\cdot\zeta}\mathcal U)_\nu,\mathcal U]
-I_Q^{\mu\nu}[U_\mu,\mathcal U_\nu,iT_{v\cdot\zeta}\mathcal U]\\
\label{Iq3}
-&I_Q^{\mu\nu}[U_\mu,(iT_{(\sqrt{1+\rho}-1)\Lambda(\zeta)}\mathcal U)_\nu,\mathcal U]
-I_Q^{\mu\nu}[U_\mu,\mathcal U_\nu,iT_{(\sqrt{1+\rho}-1)\Lambda(\zeta)}\mathcal U].
\end{align}
By (\ref{Iq}), (\ref{PkS}), (\ref{P>0C}), Lemma \ref{Sm-Lm}, Lemma \ref{Lmq=Lq}, Lemma \ref{prod-Lm}, Lemma \ref{Taf-Lp}, Lemma \ref{Eaf-Lp} and the fact that $\|(\rho,v)\|_{W^{1,\infty}}\lesssim \|(\rho,v)\|_{H^3}\lesssim \|U\|_{H^3}$ is sufficiently small,
\[
|(\ref{Iq1})|\lesssim \|U\|_X\|U\|_{H^N}\|\mathcal{Q+S+C}\|_{H^N}
\lesssim \|U\|_X^2\|U\|_{H^N}^2.
\]

For (\ref{Iq2}), by Lemma \ref{para2diff} (ii), the operator $iT_{v\cdot\zeta}$ maps real valued functions to real valued functions,
so $(iT_{v\cdot\zeta}\mathcal U)_\nu=iT_{v\cdot\zeta}U_\nu$.
Taking the complex conjugation on the third slot of $I_Q^{\mu\nu}$ into account we have
\[
-(\ref{Iq2})=\Re\frac{C^3}{R^6}\sum_{\xi,\eta,\theta\in(2\pi\Z/R)^2}
ir_{\mu,j}(\xi,\eta,\theta)\hat U_\mu(\xi-\eta-\theta)\mathcal{\hat U_\nu}(\eta)\overline{\mathcal{\hat U(\xi)}}\hat v_j(\theta),
\]
where
\begin{align*}
r_{\mu,j}(\xi,\eta,\theta)&=\frac{2\eta_j+\theta_j}2\times\frac{q_{k_1,k_2;k_3}(\xi-\eta-\theta,\eta+\theta)}{\Phi_{\mu+}(\xi-\eta-\theta,\eta+\theta)}\varphi_{\le-10}\left( \frac{|\theta|}{|2\eta+\theta|} \right)\\
&-\frac{2\xi_j-\theta_j}2\times\frac{q_{k_1,k_2;k_3}(\xi-\eta-\theta,\eta)}{\Phi_{\mu+}(\xi-\eta-\theta,\eta)}\varphi_{\le-10}\left( \frac{|\theta|}{|2\xi-\theta|} \right),\\
q_{k_1,k_2;k_3}(\xi_1,\xi_2)&=n(\xi_1,\xi_2)\varphi_{k_1}(\xi_1)\varphi_{k_2}(\xi_2)\varphi_{k_3}(\xi_1+\xi_2).
\end{align*}

Since
\begin{align*}
\left( \partial_{\xi_2}\Phi_{\mu+}^{-1} \right)(\xi-\eta-\theta,\eta+t\theta)
&=\frac{\nabla\Lambda(\xi-\theta+t\theta)-\nabla\Lambda(\eta+t\theta)}
{\Phi(\xi-\eta-\theta,\eta+t\theta)^2}\\
&=\frac{\int_0^1 (\xi-\eta-\theta)\cdot\nabla^2\Lambda(\eta+s(\xi-\eta-\theta)+t\theta)ds}
{\Phi(\xi-\eta-\theta,\eta+t\theta)^2},
\end{align*}
by Lemma \ref{Soo-Cn} (ii), for $k_1\le k_2-6$ and $k_3=k_2+O(1)$ we have
\[
\|\partial_{\xi_2}\left( \Phi_{\mu+}^{-1} \right)(\xi-\eta-\theta,\eta+t\theta)\|_{S^\infty_{k_1,k_2;k_3}}\lesssim 2^{8k_1^++k_1-k_3^+}.
\]
Using the fundamental theorem of calculus as in Lemma \ref{q-Soo} we obtain
\[
\|r_{\mu,j}(\xi,\eta,\theta)\varphi_{k_4}(\theta)\|_{S^\infty}\lesssim 2^{2Nk_3^++10k_1^++k_4^+}1_{k_1,k_4\le k_3-5}.
\]
Using Lemma \ref{paraprod} and Lemma \ref{XZ-bound} (iii) and summing over $k_1,k_4\le k_3-5$ and $k_2=k_3+O(1)$ give
\[
|(\ref{Iq2})|\lesssim \|U\|_X^2\|U\|_{H^N}^2.
\]

For (\ref{Iq3}) we distinguish two cases.

{\bf Case 1:} $\nu=+$. Then we can omit the subscript $\nu$ and get the same cancellation as in (\ref{Iq2}) from the complex conjugation in the definition of $I_Q^{\mu+}$, so by Lemma \ref{paraprod},
\[
|(\ref{Iq3})|\lesssim \|U\|_X^2\|U\|_{H^N}^2.
\]

{\bf Case 2:} $\nu=-$. On the support of $q$ we have $k_1\le k_2-6$ and $k_3\ge-1$, so Case 2.2 of Lemma \ref{1/F-Soo} gives
\[
\|\Phi_{\mu\nu}^{-1}\|_{S^\infty_{k_1,k_2;k_3}}\lesssim 2^{-k_3},
\text{ so }\|q/\Phi_{\mu\nu}\|_{S^\infty_{k_1,k_2;k_3}}\lesssim 2^{(2N-1)k_3+k_1}.
\]
This can be used to obtain the desired bound by recovering the loss of derivative in $T_{(\sqrt{1+\rho}-1)\Lambda(\zeta)}$:
\[
\|T_{(\sqrt{1+\rho}-1)\Lambda(\zeta)}\mathcal U\|_{H^{N-1}}\lesssim \|U\|_X\|U\|_{H^N}.
\]
Then Lemma \ref{paraprod} gives
\[
|(\ref{Iq3})|\lesssim \|U\|_X^2\|U\|_{H^N}^2.
\]

Combining the three bounds and integrating in time show the claim.
\end{proof}

\subsection{Quartic energy estimates}
\begin{proof}[Proof of (\ref{growth-Em-X})]
By (\ref{HN0}) and conservation of the energy $E$ (\ref{E-def}) we know that $\sup_{[0,t]}E\lesssim\ep^2$. By Proposition \ref{E-HN}, (\ref{HN0}) and (\ref{growthX1}) we know that $\mathcal E(0)\lesssim\ep^2$ and $\|P_{\ge0}U(t)\|_{H^N}^2\le\mathcal E(t)+\ep_1^3$. By Proposition \ref{E4}, Proposition \ref{Es} and Proposition \ref{Eq} we have
\[
\mathcal E(t)=\mathcal E(0)+O(\ep_1^3+\ep_1^2\ep_2^2)\lesssim\ep^2+\ep_1^3+\ep_1^2\ep_2^2.
\]
Then the same bound holds for $\|P_{\ge0}U(t)\|_{H^N}^2$, and
\[
\|U(t)\|_{H^N}^2\lesssim \|P_{\ge0}U(t)\|_{H^N}^2+E(t)\lesssim\ep^2+\ep_1^3+\ep_1^2\ep_2^2.
\]
Taking the square root gives (\ref{growth-Em-X}).
\end{proof}

\begin{proof}[Proof of (\ref{growth-Em-Z})]
Integrate Lemma \ref{dispersiveZ} (ii) in $t$ and put it in (\ref{growth-Em-X}).
\end{proof}

\section{Strichartz estimates for small $R$}\label{StrEst}
\subsection{Definition of the profile}
The evolution equation (\ref{Ut}) for $U$ can be rewritten as
\begin{align*}
(U_\sigma)_t+i\Lambda U_\sigma
&=\sum_{\mu,\nu=\pm} N^\sigma_{\mu\nu}[U_\mu,U_\nu], &
\sigma&\in\{+,-\}.
\end{align*}
Define the profile $V_\pm(t)=e^{\pm it\Lambda}U_\pm(t)$.
Then the evolution equation for $V=V_+$ is
\begin{equation}\label{dtV}
\hat V_t(\xi)=\frac{C}{R^2}\sum_{\mu,\nu=\pm} \sum_{\xi_1,\xi_2\in(2\pi\Z/R)^2\atop\xi_1+\xi_2=\xi} e^{it\Phi_{\mu\nu}(\xi_1,\xi_2)} m_{\mu\nu}(\xi_1,\xi_2)\hat V_\mu(\xi_1)\hat V_\nu(\xi_2),
\end{equation}
where $\Phi_{\mu\nu}$ is defined in (\ref{F-def}), and $m_{\mu\nu}$ are sums of terms of the form $a_0a_1$,
\begin{equation}\label{m-prod}
a_0\in\{\Lambda(\xi_1+\xi_2)|\xi_1|/\Lambda(\xi_1),
\Lambda(\xi_1+\xi_2)|\xi_2|/\Lambda(\xi_2), |\xi_1+\xi_2|\}
\end{equation}
and $a_1$ is a Calderon--Zygmund operator on $(\R^2)^2$.
By Lemma \ref{XZ-bound} (v) we can assume $a_1=1$.

Thanks to Lemma \ref{1/F-Soo}, we can integrate (\ref{dtV}) by parts in $t$ to get
\begin{align}
\label{V=W+H}
V(t)-V(0)&=\sum_{\mu,\nu=\pm} \left( W_{\mu\nu}(t)-W_{\mu\nu}(0)
-\sum_{\mu,\nu,\rho=\pm} \int_0^t H_{\mu\nu\rho}(s)ds \right),\\
\label{W2-V}
\mathcal FW_{\mu\nu}(\xi,t)&=\frac{C}{R^2}\sum_{\xi_1,\xi_2\in(2\pi\Z/R)^2\atop\xi_1+\xi_2=\xi}e^{it\Phi_{\mu\nu}(\xi_1,\xi_2)}\frac{m_{\mu\nu}(\xi_1,\xi_2)}{i\Phi_{\mu\nu}(\xi_1,\xi_2)}\hat V_\mu(\xi_1,t)\hat V_\nu(\xi_2,t),\\
\label{H3-V}
\mathcal FH_{\mu\nu\rho}(\xi,t)&=\frac{C^2}{R^4}\sum_{\xi_1,\xi_2,\xi_3\in(2\pi\Z/R)^2\atop\xi_1+\xi_2+\xi_3=\xi} e^{it\Phi_{\mu\nu}(\xi_1,\xi_2)}m_{\mu\nu\rho}(\xi_1,\xi_2,\xi_3)\hat V_\mu(\xi_1,t)\hat V_\nu(\xi_2,t)\hat V_\rho(\xi_3,t),
\end{align}
where $m_{\mu\nu\rho}$ is a linear combination of multipliers of the form $m_{\mu\sigma}m_{\nu\rho}/\Phi_{\mu\sigma}$ or $m_{\sigma\rho}m_{\mu\nu}/\Phi_{\sigma\rho}$, see (\ref{H3k}).
We need to bound the $S^\infty$ norms of such multipliers.

\begin{lemma}
For $k_1$, $k_2$, $k_3\in\Z$ we have
\begin{equation}\label{m/F-Soo}
\|m_{\mu\nu}/\Phi_{\mu\nu}\|_{S^\infty_{k_1,k_2;k_3}}\lesssim 2^{\max k_j+7\min k_j^+}.
\end{equation}
If $m_{\mu\nu\rho}=m_{\mu\sigma}m_{\nu\rho}/\Phi_{\mu\sigma}$,
then for $k$, $k_j$ and $l\in\Z$ we have
\begin{equation}\label{mk3-Soo}
\|\varphi_l(\xi_2+\xi_3)m_{\mu\nu\rho}(\xi_1,\xi_2,\xi_3)\|_{S^\infty_{k_1,k_2,k_3;k}}\lesssim 2^{2\max k_j+7\min(k_1,l)^+}
\lesssim 2^{2\max k_j+7\operatorname{med} k_j^+}.
\end{equation}
Similar bounds hold if $m_{\mu\nu\rho}=m_{\sigma\rho}m_{\mu\nu}/\Phi_{\sigma\rho}$.
\end{lemma}
\begin{proof}
(\ref{m/F-Soo}) and the first bound of (\ref{mk3-Soo}) follow from (\ref{m-prod}), Lemma \ref{Soo-Cn} and Lemma \ref{1/F-Soo}.
The second bound of (\ref{mk3-Soo}) follows from the fact that at least two of $k_j\ge\min(k_1,l)-O(1)$.
\end{proof}

\subsection{Strichartz Estimates}
In this subsection we show the Strichartz estimate (\ref{growthX2}).
\begin{proposition}\label{U-X}
Assume $N\ge M+5$ and (\ref{growthX1}). Then
\[
\|U\|_{L^2([0.t])X}\lesssim \mathcal L_R\sqrt{1+t/R}\ep_1(1+\mathcal L_R^{3/2}\ep_2^2).
\]
\end{proposition}
\begin{proof}
From (\ref{V=W+H}) it follows that
\[
U(t)=e^{-it\Lambda}\left( U(0)+\sum_{\mu,\nu=\pm} (W_{\mu\nu}(t)-W_{\mu\nu}(0))+\int_0^t \sum_{\mu\nu\rho} H_{\mu\nu\rho}(s)ds \right).
\]

{\bf Part 1:} The linear term.
By Lemma \ref{dispersiveTT*} and (\ref{growthX1}),
\[
\|e^{-is\Lambda}U(0)\|_{L^2([0,t])X}\lesssim \mathcal L_R\sqrt{1+t/R}\ep_1.
\]

{\bf Part 2:} The quadratic boundary terms. We rewrite (\ref{W2-V}) as
\[
e^{-it\Lambda(\xi)}\mathcal FW_{\mu\nu}(\xi,t)
=\frac{C}{R^2}\sum_{\xi_1,\xi_2\in(2\pi\Z/R)^2\atop\xi_1+\xi_2=\xi}
\frac{m_{\mu\nu}(\xi_1,\xi_2)}{i\Phi_{\mu\nu}(\xi_1,\xi_2)}\hat U_\mu(\xi_1,t)\hat U_\nu(\xi_2,t).
\]
We view $N^\sigma_{\mu\nu}$ and $W_{\mu\nu}$ as bilinear forms of $V_\mu$ and $V_\nu$, and decompose
\begin{align}\label{W2k}
W_{\mu\nu}&=\sum_{k_1,k_2\in\Z} W^{\mu\nu}_{k_1,k_2}, &
W^{\mu\nu}_{k_1,k_2}&=W_{\mu\nu}[P_{k_1}V_\mu,P_{k_2}V_\nu].
\end{align}
By symmetry we can assume $k_1\le k_2$. Since by (\ref{m-prod}) and (\ref{1/F-C0}),
\begin{equation}\label{m/F-C0}
|(m_{\mu\nu}/\Phi_{\mu\nu})(\xi_1,\xi_2)|\lesssim (|\xi_1|+|\xi_2|)(1+\min(|\xi_1|,|\xi_2|,|\xi_1+\xi_2|)),
\end{equation}
we have
\begin{align}
\nonumber
\|W^{\mu\nu}_{k_1,k_2}\|_{L^2}
&\lesssim 2^{k_2+k_1^+}\|\varphi_{k_1}\mathcal FU\|_{L^1}\|\varphi_{k_2}\mathcal FU\|_{L^2}\\
\label{WL2}
&\lesssim 2^{-Nk_2^++k_2+k_1^++k_1}\|P_{k_1}U\|_{L^2}\|P_{k_2}U\|_{H^N},\\
\nonumber
\|P_kW^{\mu\nu}_{k_1,k_2}\|_{L^2}
&\lesssim 2^k\|\mathcal FW^{\mu\nu}_{k_1,k_2}\|_{L^\infty}
\lesssim 2^{k_2+k_1^++k}\|P_{k_1}U\|_{L^2}\|P_{k_2}U\|_{L^2}\\
&\lesssim 2^{-Nk_2^++k_2+k_1^++k}\|P_{k_1}U\|_{L^2}\|P_{k_2}U\|_{H^N}.
\label{WL22}
\end{align}
Since $\mathcal FW^{\mu\nu}_{k_1,k_2}$ is supported on the ball $B(0,O(2^{k_2}))$, by (\ref{growthX1}) and (\ref{WL2}) we have
\[
\|W^{\mu\nu}_{k_1,k_2}\|_{H^{M+2}}\lesssim 2^{(M+2-N)k_2^++k_2-|k_1|}\ep_1^2.
\]
Summing over $k_1,k_2\in\Z$ gives $\|W_{\mu\nu}\|_{H^{M+2}}\lesssim \ep_1^2$. Then by Lemma \ref{dispersiveTT*},
\[
\|e^{-is\Lambda}W_{\mu\nu}(0)\|_{L^2([0,t])X}
\lesssim \mathcal L_R\sqrt{1+t/R}\ep_1^2.
\]

By the Bernstein inequality, Lemma \ref{paraprod} ($p_1=\infty$, $p_2=2$) and (\ref{m/F-Soo}),
\[
\|P_ke^{-is\Lambda}W^{\mu\nu}_{k_1,k_2}(s)\|_{L^\infty}
\lesssim 2^k\|P_kW^{\mu\nu}_{k_1,k_2}(s)\|_{L^2}
\lesssim 2^{k+k_2+7k_1^+}\|P_{k_1}U(s)\|_{L^\infty}\|P_{k_2}U(s)\|_{L^2}.
\]
Again using the support of $\mathcal FW^{\mu\nu}_{k_1,k_2}$ and
(\ref{growthX1}) we get $\|P_ke^{-is\Lambda}W_{\mu\nu}(s)\|_{L^\infty}\lesssim 2^{k+(1-N)k^+}\|U(s)\|_X\ep_1$. Then $\|e^{-is\Lambda}W_{\mu\nu}(s)\|_X\lesssim \|U(s)\|_X\ep_1$. Taking the $L^2$ norm in $t$ gives
\[
\|e^{-is\Lambda}W_{\mu\nu}(s)\|_{L^2([0,t])X}\lesssim \ep_1\ep_2.
\]

{\bf Part 3:} The cubic bulk terms. We write (\ref{H3-V}) as
\[
\mathcal FH_{\mu\nu\rho}(\xi,t)=\frac{C^2}{R^4}e^{it\Lambda(\xi)}
\sum_{\xi_1,\xi_2,\xi_3\in(2\pi\Z/R)^2\atop\xi_1+\xi_2+\xi_3=\xi} m_{\mu\nu\rho}(\xi_1,\xi_2,\xi_3)\hat U_\mu(\xi_1,t)\hat U_\nu(\xi_2,t)\hat U_\rho(\xi_3,t),
\]
view it as a trilinear form $H_{\mu\nu\rho}=H_{\mu\nu\rho}[V_\mu,V_\nu,V_\rho]$ and decompose
\begin{equation}\label{H3k}
\begin{aligned}
H_{\mu\nu\rho}&=\sum_{l\in\Z,\sigma=\pm} H^{\mu\nu\rho\sigma}_l,\\
H^{\mu\nu\rho\sigma}_l[V_\mu,V_\nu,V_\rho](t)
&=W_{\mu\sigma}[V_\mu,P_le^{it\sigma\Lambda}N_{\nu\rho}^\sigma[V_\nu,V_\rho]]\\
&+W_{\sigma\rho}[P_le^{it\sigma\Lambda}N_{\mu\nu}^\sigma[V_\mu,V_\nu],V_\rho]],\\
H^{\mu\nu\rho\sigma}_l&=\sum_{k_1,k_2,k_3\in\Z} H^{\mu\nu\rho\sigma}_{k_1,k_2,k_3,l},\\
H^{\mu\nu\rho\sigma}_{k_1,k_2,k_3,l}
&=H^{\mu\nu\rho\sigma}_l[P_{k_1}V_\mu,P_{k_2}V_\nu,P_{k_3}V_\rho].
\end{aligned}
\end{equation}
Assume $k_1\le k_2\le k_3$. By Lemma \ref{paraprod} ($p_1=p_2=\infty$, $p_3=2$) and (\ref{mk3-Soo}),
\[
\|P_kH^{\mu\nu\rho\sigma}_{k_1,k_2,k_3,l}\|_{L^2}
\lesssim 2^{2k_3+7k_2^+}\|P_{k_1}U\|_{L^\infty}\|P_{k_2}U\|_{L^\infty}\|P_{k_3}U\|_{L^2}.
\]
Summing over $k_1,k_2\in\Z$, $-\mathcal L_R\le l\le k_3+O(1)$ and $k_3\ge k-O(1)$ we get
\[
\|P_kH_{\mu\nu\rho}(t)\|_{L^2}\lesssim (\mathcal L_R+k^++1)2^{(2-N)k^+}\|U(t)\|_X^2\|U(t)\|_{H^N}.
\]
Taking an $\ell^2$ sum in $k\ge-\mathcal L_R$ and using (\ref{growthX1}) and $N\ge M+5$ we get
\[
\|H_{\mu\nu\rho}(t)\|_{H^{M+2}}\lesssim \mathcal L_R^{3/2}\|U(t)\|_X^2\ep_1.
\]
By Lemma \ref{dispersiveTT*} and Lemma \ref{XZ-bound} (ii) we then have
\begin{align*}
\left\| e^{-is\Lambda} \int_0^s H_{\mu\nu\rho}(s')ds' \right\|_{L^2_s([0,t])X}
&\le\int_0^t \|e^{-is\Lambda}H_{\mu\nu\rho}(s')\|_{L^2_s([s',t])X}ds'\\
&\lesssim \mathcal L_R\sqrt{1+t/R}\|H_{\mu\nu\rho}\|_{L^1([0,t])H^{M+2}}\\
&\lesssim \mathcal L_R^{5/2}\sqrt{1+t/R}\ep_1\ep_2^2.
\end{align*}

Combining Part 1 through Part 3 shows the claim.
Note that Part 2 is dominated by Part 1 and Part 3.
\end{proof}

For future use we also need to bound $H^{\mu\nu\rho\sigma}_l$ for small $l$.
\begin{proposition}\label{PkH-L2-l-small}
Assume $N\ge M+3$ and (\ref{growthX1}). Then for $l\le0$ we have
\[
\|H^{\mu\nu\rho\sigma}_l\|_{H^{M+2}}\lesssim 2^{2l}\ep_1^3.
\]
\end{proposition}
\begin{proof}
We use $|m_{\mu\nu}(\xi_1,\xi_2)|\lesssim |\xi_1|+|\xi_2|$ to get
\[
\|\mathcal FN_{\nu\rho}^\sigma[V_\nu,V_\sigma]\|_{L^\infty}
\lesssim \|V\|_{H^1}^2\lesssim \ep_1^2.
\]
When $|\xi_2|\lesssim 1$, by (\ref{m/F-C0}) we also have
$|(m_{\mu\nu}/\Phi_{\mu\nu})(\xi_1,\xi_2)|\lesssim \Lambda(\xi_1)$, so
\begin{align*}
\|W_{\mu\sigma}[V_\mu,P_le^{it\Lambda}N_{\nu\rho}^\sigma[V_\nu,V_\rho]]\|_{H^{M+2}}
&\lesssim \||\mathcal F\Lambda^{M+3}V|*|\mathcal FP_lN_{\nu\rho}^\sigma[V_\nu,V_\sigma]|\|_{L^2}\\
&\lesssim \|V\|_{H^{M+3}}\|\mathcal FP_lN_{\nu\rho}^\sigma[V_\nu,V_\sigma]\|_{L^1}\\
&\lesssim 2^{2l}\ep_1^3
\end{align*}
and a similar bound holds for $\|P_kW_{\sigma\rho}[P_lN_{\mu\nu}^\sigma[V_\mu,V_\nu],V_\rho]]\|_{H^{M+2}}$.
\end{proof}

\section{$Z$-norm estimates for large $R$}\label{ZEst}
This section is devoted to the proof of the $Z$ norm estimate (\ref{growthZ2}),
which is contained in Proposition \ref{W-Z} and Proposition \ref{H-Z} below.

\subsection{Integration by parts in phase space}
We need a lemma to integrate by parts in phase space.
\begin{lemma}[\cite{GuIoPa}, Lemma A.2 or \cite{IoPa2}, Lemma 5.4]\label{int-part}
Let $0<\ep\le1/\ep\le K$. Suppose $f$, $g:\R^d\to\R$ satisfies
\[
|\nabla f|\ge 1_{\supp g},\text{ and for all }L\ge2,\ 
|\nabla^L f|\lesssim_L \ep^{1-L}\text{ on }\supp g.
\]
Then
\[
\left| \int e^{iKf}g \right|\lesssim_{d,L} (K\ep)^{-L} \sum_{l=0}^L \ep^l\|\nabla^lg\|_{L^1}.
\]
\end{lemma}

We will only use the case when $\ep=1$, for which the bound reads
\[
\left| \int e^{iKf}g \right|\lesssim_{d,L} K^{-L} \|g\|_{W^{L,1}}.
\]

\begin{lemma}\label{int-part-cor}
Let $r\ge2$. Suppose $f$, $g:\R^d\to\R$ satisfies $|\nabla f|\le r/2$ on $\supp g$,
and for all $L\ge2$, we have $|\nabla^Lf|\lesssim_L r$ on $\supp g$. Let
\[
\mathcal K(x)=\int e^{i(x\cdot\xi+f(\xi))} g(\xi)d\xi.
\]
Then for all $L\ge 1$ we have
\[
\|\mathcal K\|_{L^1(\R^d\backslash B(0,r))}\lesssim_L r^{-L} \|g\|_{W^{L+d,1}}.
\]
\end{lemma}
\begin{proof}
We have
\[
x\cdot\xi+f(\xi)=\frac{|x|}2\left( \frac{2x}{|x|}\cdot\xi+\frac{2f(\xi)}{|x|} \right).
\]
Suppose $|x|\ge r$. Let $K=|x|/2\ge r/2\ge1$, and
\[
F(\xi)=\frac{2x}{|x|}\cdot\xi+\frac{2f(\xi)}{|x|}.
\]
Then $|\nabla F|\ge1$ and $|\nabla^LF|\lesssim_L 1$. By Lemma \ref{int-part},
for $|x|\ge r$ we have
\[
|\mathcal K(x)|\lesssim_L |x|^{-L}\|g\|_{W^{L,1}}.
\]
The result follows from integrating this bound with $L+d$ in place of $L$.
\end{proof}

\subsection{Bounding the quadratic boundary terms}
In this section we bound the $Z$ norm of the quadratic boundary terms $W_{\mu\nu}$.
\begin{proposition}\label{W-Z}
Assume $N\ge3(M+4)$ and (\ref{growthZ}). Then
\[
\|W_{\mu\nu}(t)\|_Z\lesssim (t^{1+9/N}/R^{4/3}+1)\ep_1^2.
\]
\end{proposition}
\begin{proof}
We use the decomposition (\ref{W2k}) and assume by symmetry $k_1\le k_2$
(except in Case 4.1 below). We distinguish several cases to estimate
\[
\|W_{\mu\nu}\|_Z\approx\left\| 2^{2j/3}\|Q_j\Lambda^{M+2}W_{\mu\nu}\|_{L^2} \right\|_{\ell^2_{0\le j\le\log R+1}}
\]

{\bf Case 1:} $k_2\ge j/N$. We sum (\ref{WL2}) over $k_1\in\Z$ and $k_2\ge j/N$, and use $(M+3-N)k_2\le-2Nk_2/3-k_2\le-2j/3-j/N$ to get
\[
\left\| \sum_{k_1\in\Z,k_2\ge j/N} 2^{2j/3}\|Q_j\Lambda^{M+2}W^{\mu\nu}_{k_1,k_2}\|_{L^2} \right\|_{\ell^2_{j\ge0}}
\lesssim \sum_{j\ge0} 2^{-j/N}\|U\|_{H^N}^2\lesssim \ep_1^2.
\]

{\bf Case 2:} $k_1\le-3j/4$. We sum (\ref{WL2}) with $k_1\le-3j/4$ and $k_2\in\Z$ and use $N\ge M+4$ to get the same bound as Case 1.

{\bf Case 3:} $P_{k_3}W_{\mu\nu}$ for $k_3\le-3j/4<k_1$.
We sum (\ref{WL22}) with $k_3$ in place of $k$, $k_1\in[k_3,k_2]$,
$k_3\le-3j/4$ and $k_2\in\Z$ to get the same bound as Case 1.

{\bf Case 4:} $-3j/4<k_i\le j/N+O(1)$, $1\le i\le 3$.

{\bf Case 4.1:} $j\le\mathcal L+5$. We decompose
\begin{align*}
W^{\mu\nu}_{k_1,k_2}&=\sum_{j_1,j_2\ge0} W^{\mu\nu}_{j_1,k_1,j_2,k_2},\\
W^{\mu\nu}_{j_1,k_1,j_2,k_2}&=W_{\mu\nu}[P_{[k_1-1,k_1+1]}Q_{j_1}P_{k_1}V_\mu,P_{[k_2-1,k_2+1]}Q_{j_2}P_{k_2}V_\nu].
\end{align*}
We now assume $j_1\le j_2$ instead of $k_1\le k_2$.
By Lemma \ref{paraprod}, (\ref{m/F-Soo}), Lemma \ref{dispersive},
unitarity of $e^{it\Lambda}$, H\"older's inequality and Lemma \ref{XZ-bound} (iv) we have
\begin{align*}
\|\Lambda^{M+2}P_{k_3}W^{\mu\nu}_{j_1,k_1,j_2,k_2}(t)\|_{L^2}
&\lesssim 2^{(M+3)K^++7k_2^+}\|e^{-it\mu\Lambda}P_{[k_1-1,k_1+1]}Q_{j_1}P_{k_1}V_\mu(t)\|_{L^\infty}\|e^{-it\nu\Lambda}Q_{j_2}P_{k_2}V_\nu(t)\|_{L^2}\\
&\lesssim 2^{(M+3)K^++7k_2^++2k_1^+}\frac{(t/R+1)^2}{1+t} \|Q_{j_1}P_{k_1}V(t)\|_{L^1}
\|Q_{j_2}P_{k_2}V(t)\|_{L^2}\\
&\lesssim 2^{8K^++j_1/3-2j_2/3}\frac{(t/R+1)^2}{1+t}\ep_1^2,
\end{align*}
where we recall $K=\max k_i$. We sum over $j_1,j_2\in\Z$, $j_1\le j_2$ and $-\mathcal L-5<-3j/4<k_i\le j/N+O(1)\le\mathcal L/N+O(1)$ to get
\[
\sum_{-2j/3<k_i\le j/N+O(1)} \|\Lambda^{M+2}P_{k_3}W^{\mu\nu}_{k_1,k_2}(t)\|_{L^2}
\lesssim \frac{(t/R+1)^2}{(1+t)^{1-9/N}}\ep_1^2.
\]
Then we sum over $0\le j\le\min(\mathcal L,\log R)+5$ and use $N\ge27$ to get
\begin{align*}
&\left\| \sum_{-3j/4<k_i\le j/N+O(1)} 2^{2j/3}\|Q_j\Lambda^{M+2}P_{k_3}W^{\mu\nu}_{k_1,k_2}(t)\|_{L^2} \right\|_{\ell^2_{0\le j\le\min(\mathcal L,\log R)+5}}\\
\lesssim&(t^{1+9/N}/R^{4/3}+1)\ep_1^2.
\end{align*}

{\bf Case 4.2:} $j>\mathcal L+5$. In this case $t<2^{j-5}$. We decompose
\begin{align*}
W^{\mu\nu}_{k_1,k_2}&=A^{\mu\nu}_{k_1,k_2}+B^{\mu\nu}_{k_1,k_2},\\
A^{\mu\nu}_{k_1,k_2}
&=W_{\mu\nu}[P_{[k_1-1,k_1+1]}Q_{\ge j-4}P_{k_1}V_\mu,P_{k_2}V_\nu],\\
B^{\mu\nu}_{k_1,k_2}
&=W_{\mu\nu}[P_{[k_1-1,k_1+1]}Q_{\le j-5}P_{k_1}V_\mu,P_{k_2}V_\nu],
\end{align*}
where we have used $P_{[k-1,k+1]}P_k=P_k$.

For $A$ we have, by (\ref{WL2}) and unitarity of $e^{it\Lambda}$,
\[
\|\Lambda^{M+2}A^{\mu\nu}_{k_1,k_2}\|_{L^2}
\lesssim 2^{(M+2-N)k_2^++k_2+k_1^++k_1}\|Q_{\ge j-4}P_{k_1}V\|_{L^2}\|P_{k_2}U\|_{H^N}.
\]
We sum over $j\ge0$, $k_1$, $k_2\in\Z$ and use $N\ge M+4$ and Lemma \ref{XZ-bound} (iv) to get
\[
\sum_{k_1,k_2\in\Z} \|2^{2j/3} \|Q_j\Lambda^{M+2}A^{\mu\nu}_{k_1,k_2}\|_{L^2}\|_{\ell^2_{j\ge0}}
\lesssim \|V\|_Z\|U\|_{H^N}\lesssim \ep_1^2.
\]

To bound $B$, we can assume the support of $Q_j$ intersects the torus.
This implies that $2^{j-1}<R/\sqrt2$, or $R>2^{j-1/2}$. We write
\begin{align*}
\Lambda^{M+2}P_{k_3}B^{\mu\nu}_{k_1,k_2}(x,t)
&=\int G(x,y,z,t)Q_{\le j-5}P_{k_1}V_\mu(y,t)P_{k_2}V_\nu(z,t)dydz,\\
G(x,y,z,t)&=\frac{C^2}{R^4}\sum_{\xi_1,\xi_2\in(2\pi\Z/R)^2} e^{i\phi_{\mu\nu}(\xi_1,\xi_2)}\Lambda(\xi_1+\xi_2)^{M+2}m^{\mu\nu}_{k_1,k_2;k_3}(\xi_1,\xi_2),\\
\phi_{\mu\nu}(\xi_1,\xi_2)&=(x-y)\cdot\xi_1+(x-z)\cdot\xi_2+t\Phi_{\mu\nu}(\xi_1,\xi_2),\\
m^{\mu\nu}_{k_1,k_2;k_3}(\xi_1,\xi_2)&=\varphi_{[k_1-1,k_1+1]}(\xi_1)\varphi_{[k_2-1,k_2+1]}(\xi_2)\frac{m_{\mu\nu}(\xi_1,\xi_2)}{i\Phi_{\mu\nu}(\xi_1,\xi_2)}\varphi_{k_3}(\xi_1+\xi_2).
\end{align*}
By the Poisson summation formula we have
\begin{align}
\label{G-K}
G(x,y,z,t)&=\sum_{y',z'\in(R\Z)^2} \mathcal K(x,y+y',z+z',t),\\
\nonumber
\mathcal K(x,y,z,t)&=\int e^{i\phi_{\mu\nu}(\xi_1,\xi_2)}\Lambda(\xi_1+\xi_2)^{M+2}m^{\mu\nu}_{k_1,k_2;k_3}(\xi_1,\xi_2)d\xi_1 d\xi_2.
\end{align}

Since $m_{\mu\nu}=a_0$ is of the form (\ref{m-prod}), for $L\ge0$ we have
\begin{equation}\label{m-CL}
|\nabla^L m_{\mu\nu}(\xi_1,\xi_2)|\lesssim_L (|\xi_1|+|\xi_2|)(1+\min(|\xi_1|,|\xi_2|,|\xi_1+\xi_2|)^{-L}).
\end{equation}
From (\ref{m-CL}), $|\nabla^L(\varphi_k\Lambda^M)|\lesssim_L 2^{Mk^+-Lk}$ and (\ref{1/F-CL}) it follows that
\[
\left| \nabla^L \left( \Lambda(\xi_1+\xi_2)^{M+2}m^{\mu\nu}_{k_1,k_2;k_3}(\xi_1,\xi_2) \right) \right|
\lesssim_L 2^{(M+2)k_2^++k_2}(2^{\min k_i}+2^{-L\min k_i}).
\]
Using $-3j/4<k_i\le j/N+O(1)$ we have
\[
\left\| \Lambda(\xi_1+\xi_2)^{M+2}m^{\mu\nu}_{k_1,k_2;k_3}(\xi_1,\xi_2) \right\|_{W^{24,1}}\lesssim 2^{(M+2)k_2^++3k_2+2k_3+18j}.
\]
Since $|\nabla\Lambda|<1$, we have $|t\nabla\Phi_{\mu\nu}(\xi_1,\xi_2)|<4t<2^{j-3}$, so by Lemma \ref{int-part-cor} (with $L=20$ and $d=4$),
\begin{equation}\label{K2-L1}
\|1_{|x-y|>2^{j-2}}\mathcal K(x,y,z,t)\|_{L^1_{y,z}}\lesssim 2^{(M+2)k_2^++3k_2+2k_3-2j}.
\end{equation}
When $\varphi_j(x)\varphi_{\le j-5}(y)>0$, we have $|x-y|>2^{j-2}$,
and for all $y'\in(R\Z)^2\backslash\{(0,0)\}$ we have $|x-y-y'|>R/2-2^{j-4}>2^{j-2}$. Then by (\ref{G-K}),
\[
\|\varphi_j(x)\varphi_{\le j-5}(y)G(x,y,z,t)\|_{L^1_{y,z}((\R/R\Z)^2)}
\lesssim\text{right-hand side of (\ref{K2-L1})}.
\]
Combining this with Bernstein's inequality $\|P_kV\|_{L^\infty}\lesssim 2^k\|P_kV\|_{L^2}=2^k\|P_kU\|_{L^2}$ we get
\begin{align*}
\|Q_j\Lambda^{M+2}P_{k_3}B^{\mu\nu}_{k_1,k_2}\|_{L^2}
&\lesssim 2^j\|Q_j\Lambda^{M+2}P_{k_3}B^{\mu\nu}_{k_1,k_2}\|_{L^\infty}\\
&\lesssim 2^{(M+2)k_2^++4k_2+2k_3+k_1-j}\|P_{k_1}U\|_{L^2}\|P_{k_2}U\|_{L^2}.
\end{align*}
We sum over $k_i\in\Z$ and $j>\mathcal L+5$ and use $N\ge M+9$ to get
\[
\left\| \sum_{-3j/4<k_i\le j/N+O(1)} 2^{2j/3}\|Q_j\Lambda^{M+2}P_{k_3}B^{\mu\nu}_{k_1,k_2}\|_{L^2} \right\|_{\ell^2_{j>\mathcal L+5}}
\lesssim \|U\|_{H^N}^2\le \ep_1^2.
\]

Combining Case 1 through Case 4 above shows Proposition \ref{W-Z}.
\end{proof}

\subsection{Bounding the cubic bulk terms}
In this section we bound the $Z$ norm of the cubic bulk terms $H_{\mu\nu\rho}$.
\begin{proposition}\label{H-Z}
Assume $N\ge\max(3(M+4),106)$ and (\ref{growthZ}). Then
\[
\int_0^t \|H_{\mu\nu\rho}(s)\|_Zds\lesssim (t^{3+33/N}/R^{10/3-2/N}+1)\ep_1^3.
\]
\end{proposition}
\begin{proof}
Recall (\ref{H3k}). We assume by symmebtry $k_1\le k_2\le k_3$ (except in Case 5.1 below). From (\ref{m-prod}), (\ref{m/F-C0}) and (\ref{1/F-C0}) it follows in the same way as (\ref{WL2}) that
\begin{equation}\label{H3k-L2}
\begin{aligned}
\|\Lambda^{M+2}H^{\mu\nu\rho}_{k_1,k_2,k_3}\|_{L^2}
&+\|\Lambda^{M+2}H^{\mu\nu\rho}_{k_1,k_2,k_3,l}\|_{L^2}\\
&\lesssim 2^{k_1+k_2+(M+2)k_3^++2k_3+k_2^+}\|P_{k_1}U\|_{L^2}\|P_{k_2}U\|_{L^2}\|P_{k_3}U\|_{L^2}.
\end{aligned}
\end{equation}

We distinguish several cases.

{\bf Case 1:} $k_3\ge3\max(j,\mathcal L)/N$. We sum (\ref{H3k-L2}) over $k_1,k_2\in\Z$ and $k_3\ge3\max(j,\mathcal L)/N$, and use $(M+4-N)k_3\le-2Nk_3/3\le-2\max(j,\mathcal L)$ to get
\begin{align*}
\left\| \sum_{k_1,k_2\in\Z\atop k_3\ge3\max(j,\mathcal L)/N} 2^{2j/3}\|Q_j\Lambda^{M+2}H^{\mu\nu\rho}_{k_1,k_2,k_3}(t)\|_{L^2} \right\|_{\ell^2_{j\ge0}}
&\lesssim \sum_{j\ge0} 2^{2j/3-2\max(j,\mathcal L)}\|U(t)\|_{H^N}^3\\
&\lesssim (1+t)^{-4/3}\ep_1^3.
\end{align*}

{\bf Case 2:} $l\le-6\max(j,\mathcal L)/7$. By Proposition \ref{PkH-L2-l-small} we have
\begin{align*}
\left\| \sum_{l\le-6\max(j,\mathcal L)/7} 2^{2j/3}\|Q_j\Lambda^{M+2}H^{\mu\nu\rho\sigma}_l(t)\|_{L^2} \right\|_{\ell^2_{j\ge0}}
&\lesssim \sum_{j\ge0} 2^{2j/3-12\max(j,\mathcal L)/7}\ep_1^3\\
&\lesssim (1+t)^{-1.04}\ep_1^3.
\end{align*}

{\bf Case 3:} $k_1\le-3\max(j,{\mathcal L})/4$, $k_3<3\max(j,\mathcal L)/N$ and $l>-6\max(j,\mathcal L)/7$ (so $|l|\lesssim j+\mathcal L+k_3^++1\lesssim |k_1|$). By unitarity of $e^{it\Lambda}$, Bernstein's inequality, Lemma \ref{XZ-bound} (i) and (\ref{growthX1}),
\[
\|P_kU\|_{L^2}=\|P_kV\|_{L^2}\lesssim 2^{3k/5}\|P_kV\|_{L^{5/4}}
\lesssim 2^{3k/5}\ep_1.
\]
Using in addition Lemma \ref{paraprod}, (\ref{mk3-Soo}), Lemma \ref{dispersiveZ} (i) and $M\ge7$ we get
\begin{align*}
\|\Lambda^{M+2}P_kH^{\mu\nu\rho\sigma}_{k_1,k_2,k_3,l}(t)\|_{L^2}
&\lesssim 2^{(M+2)k^++k/20}\|P_kH^{\mu\nu\rho\sigma}_{k_1,k_2,k_3,l}(t)\|_{L^{40/21}}\\
&\lesssim 2^{(M+2)k_3^++2k_3+7k_2^++k/20}\|P_{k_1}U(t)\|_{L^{40}}\|P_{k_2}U(t)\|_{L^\infty}\|P_{k_3}U(t)\|_{L^2}\\
&\lesssim 2^{(M+2-N)k_3^++2k_3+k/20+8k_1/5}(1+t)^{-19/30}(t/R+1)^{4/3}\ep_1^3.
\end{align*}
We sum over $k_2\in[k_1,k_3]\cap\Z$, $k\le k_3+O(1)$, $k_3\in\Z$, $|l|\lesssim|k_1|$, $k_1\le-3\max(j,\mathcal L)/4$ and $j\ge0$, and use $N\ge M+5$ to get
\begin{align*}
&\left\| \sum_{k_1\le-3\max(j,\mathcal L)/4} \sum_{k_2,k_3\ge k_1} \sum_{|l|\lesssim k_1} 2^{2j/3}\|Q_j\Lambda^{M+2}H^{\mu\nu\rho\sigma}_{k_1,k_2,k_3,l}(t)\|_{L^2} \right\|_{\ell^2_{j\ge0}}\\
\lesssim &\left\| \sum_{k_1\le-3\max(j,\mathcal L)/4} |k_1|^22^{2j/3+8k_1/5} \right\|_{\ell^2_{j\ge0}}\frac{(t/R+1)^{4/3}}{(1+t)^{19/30}}\ep_1^3\\
\lesssim &\sum_{j\ge0} 2^{2j/3-7\max(j,\mathcal L)/6}\frac{(t/R+1)^{4/3}}{(1+t)^{19/30}}\ep_1^3
\lesssim \frac{(t/R+1)^{4/3}}{(1+t)^{17/15}}\ep_1^3.
\end{align*}

{\bf Case 4:} $k\le-6\max(j.\mathcal L)/7$, and $m_{\mu\sigma}$ or $m_{\sigma\rho}=|\xi_1+\xi_2+\xi_3|$. Then from (\ref{m/F-C0}) it follows in the same way as (\ref{WL22}) that
\[
\|\Lambda^{M+2}P_kH^{\mu\nu\rho}_{k_1,k_2,k_3}\|_{L^2}
\lesssim 2^{2k+k_1+k_3+k_3^+}\|P_{k_1}U\|_{L^2}\|P_{k_2}U\|_{L^2}\|P_{k_3}U\|_{L^2}.
\]
We sum over $k_i\in\Z$, $k\le-6\max(j.\mathcal L)/7$ to get the same bound as Case 2.

{\bf Case 5:} $-6\max(j,\mathcal L)/7<k_i$, $l\le3\max(j,\mathcal L)/N+O(1)$, $1\le i\le 3$,
and if $m_{\mu\sigma}$ or $m_{\sigma\rho}=|\xi_1+\xi_2+\xi_3|$,
then $k>-6\max(j,\mathcal L)/7$.

{\bf Case 5.1:} $j\le\mathcal L+5$. We decompose
\[
H^{\mu\nu\rho\sigma}_{k_1,k_2,k_3,l}=\sum_{j_1,j_2,j_3} H^{\mu\nu\rho\sigma}_{j_1,k_1,j_2,k_2,j_3,k_3,l}
\]
as in Case 4.1 in the proof of Proposition \ref{W-Z}.
We now assume $j_1\le j_2\le j_3$ instead of $k_1\le k_2\le k_3$.
By Lemma \ref{paraprod}, (\ref{mk3-Soo}), Lemma \ref{dispersive},
unitarity of $e^{it\Lambda}$, H\"older's inequality and Lemma \ref{XZ-bound} (iv) we have (recall $K=\max(k,k_1,k_2,k_3)$)
\begin{align*}
&\|\Lambda^{M+2}P_kH^{\mu\nu\rho\sigma}_{j_1,k_1,j_2,k_2,j_3,k_3,l}(t)\|_{L^2}\lesssim 2^{k/N}\|\Lambda^{M+2}P_kH^{\mu\nu\rho\sigma}_{j_1,k_1,j_2,k_2,j_3,k_3,l}(t)\|_{L^{2N/(N+1)}}\\
\lesssim&2^{k/N+(M+11)K^+}\|e^{-it\mu\Lambda}P_{[k_1-1,k_1+1]}Q_{j_1}P_{k_1}V_\mu(t)\|_{L^{2N}}\|e^{-it\nu\Lambda}P_{[k_2-1,k_2+1]}Q_{j_2}P_{k_2}V_\nu(t)\|_{L^\infty}\\
\times&\|e^{-it\rho\Lambda}Q_{j_3}P_{k_3}V_\rho(t)\|_{L^2}\\
\lesssim&2^{k/N+(M+11)K^++2k_1^++2k_2^+}\frac{(t/R+1)^{4-2/N}}{(1+t)^{2-1/N}}\|Q_{j_1}P_{k_1}V(t)\|_{L^{2N/(2N-1)}}\|Q_{j_2}P_{k_2}V(t)\|_{L^1}\|Q_{j_3}P_{k_3}V(t)\|_{L^2}\\
\lesssim&2^{k/N+11K^++(j_1+j_2-2j_3)/3}\frac{(t/R+1)^{4-2/N}}{(1+t)^{2-1/N}}v_{j_1,k_1}(t)v_{j_2,k_2}(t)v_{j_3,k_3}(t),
\end{align*}
where $v_{j,k}=2^{(M+2)k^++2j/3}\|Q_jP_kV\|_{L^2}$.
By the AM-GM inequality and Lemma \ref{XZ-bound} (iii),
\begin{align*}
v_{j_1,k_1}v_{j_2,k_2}v_{j_3,k_3}&\lesssim v_{j_1,k_1}^3+v_{j_2,k_2}^3+v_{j_3,k_3}^3,\\
\|v_{j,k}\|_{\ell^3_j}&\le\|v_{j,k}\|_{\ell^2_j}\lesssim\|V\|_Z\lesssim\ep_1.
\end{align*}
We sum over $j_1,j_2,j_3\ge0$ to get
\[
\|\Lambda^{M+2}P_kH^{\mu\nu\rho\sigma}_{k_1,k_2,k_3,l}(t)\|_{L^2}
\lesssim 2^{k/N+11K^+}\frac{(t/R+1)^{4-2/N}}{(1+t)^{2-1/N}}\ep_1^3.
\]
We then sum over $-\mathcal L-5<k_i$, $l\le3\mathcal L/N+O(1)$ and $k\le3\mathcal L/N+O(1)$ to get
\[
\sum_{-\mathcal L-5<k,k_i,l\le3\mathcal L/N+O(1)}
\|\Lambda^{M+2}P_kH^{\mu\nu\rho\sigma}_{k_1,k_2,k_3,l}(t)\|_{L^2}
\lesssim \frac{(t/R+1)^{4-2/N}}{(1+t)^{2-35/N}}\ep_1^3.
\]
Then we sum over $j\le\min(\mathcal L,\log R)+5$ and use $N\ge106$ to get
\begin{align*}
&\left\| \sum_{-\mathcal L-5<k,k_i,l\le3\mathcal L/N+O(1)} 2^{2j/3}\|Q_j\Lambda^{M+2}P_kH^{\mu\nu\rho\sigma}_{k_1,k_2,k_3,l}(t)\|_{L^2} \right\|_{\ell^2_{0\le j\le\mathcal L_R+5}}\\
\lesssim&(t^{2+33/N}/R^{10/3-2/N}+(1+t)^{-1.003})\ep_1^3.
\end{align*}

{\bf Case 5.2:} $j>\mathcal L+5$. We decompose
\begin{align*}
H^{\mu\nu\rho\sigma}_{k_1,k_2,k_3,l}
&=A^{\mu\nu\rho\sigma}_{k_1,k_2,k_3,l}+B^{\mu\nu\rho\sigma}_{k_1,k_2,k_3,l},\\
A^{\mu\nu\rho\sigma}_{k_1,k_2,k_3,l}
&=H^{\mu\nu\rho\sigma}_l[P_{[k_1-1,k_1+1]}Q_{\ge j-4}P_{k_1}V_\mu,P_{[k_2-1,k_2+1]}Q_{\ge j-4}P_{k_2}V_\nu,P_{[k_3-1,k_3+1]}Q_{\ge j-4}P_{k_3}V_\rho],\\
B^{\mu\nu\rho\sigma}_{k_1,k_2,k_3,l}
&=\sum_{I_1,I_2,I_3\in\{\ge j-4,\le j-5\}\atop\exists I_i=``\le j-5"}
B^{\mu\nu\rho\sigma}_{I_1,k_1,I_2,k_2,I_3,k_3,l},\\
B^{\mu\nu\rho\sigma}_{I_1,k_1,I_2,k_2,I_3,k_3,l}
&=H^{\mu\nu\rho\sigma}_l[P_{[k_1-1,k_1+1]}Q_{I_1}P_{k_1}V_\mu,
P_{[k_2-1,k_2+1]}Q_{I_2}P_{k_2}V_\nu,P_{[k_3-1,k_3+1]}Q_{I_3}P_{k_3}V_\rho].
\end{align*}

For $A$ we have, by (\ref{H3k-L2}), unitarity of $e^{it\Lambda}$ and Lemma \ref{XZ-bound} (iv),
\begin{align*}
\|\Lambda^{M+2}A^{\mu\nu\rho\sigma}_{k_1,k_2,k_3,l}\|_{L^2}
&\lesssim 2^{k_1+k_2+(M+3)k_3^++2k_3}\|Q_{\ge j-3}P_{k_1}V\|_{L^2}\|Q_{\ge j-3}P_{k_2}V\|_{L^2}\|Q_{\ge j-3}P_{k_3}V\|_{L^2}\\
&\lesssim 2^{-(M+2)(k_1^++k_2^+)+k_3^++k_1+k_2+2k_3-2j}\ep_1^3.
\end{align*}
We sum over $k_1$, $k_2\in\Z$, $-6j/7<k_3$, $l\le3j/N+O(1)$ and $j>\mathcal L+5$, and use $N\ge28$ to get
\begin{align*}
\left\| \sum_{-6j/7<k_i,l\le3j/N+O(1)}
2^{2j/3}\|Q_j\Lambda^{M+2}A^{\mu\nu\rho\sigma}_{k_1,k_2,k_3,l}(t)\|_{L^2} \right\|_{\ell^2_{j>\mathcal L+5}}
&\lesssim\sum_{j>\mathcal L+5} j2^{(9/N-4/3)j}\ep_1^3\\
&\lesssim (1+t)^{-1.01}\ep_1^3.
\end{align*}

To bound $B$, we assume $R>2^{j-1/2}$ as before, and write
\begin{align*}
&\sum_{\mu,\nu,\rho,\sigma=\pm}
\Lambda^{M+2}B^{\mu\nu\rho\sigma}_{I_1,k_1,I_2,k_2,I_3,k_3,l}(x,t)
=\sum_{\mu,\nu,\rho=\pm}\\
&\int G(x,y,z,w,t)Q_{I_1}P_{k_1}V_\mu(y,t)
Q_{I_2}P_{k_2}V_\nu(z,t)Q_{I_3}P_{k_3}V_\rho(w,t)dydydw,\\
&G(x,y,z,w,t)=\sum_{y',z',w'\in(R\Z)^2} \mathcal K(x,y+y',z+z',w+w',t),\\
&\mathcal K(x,y,z,w,t)=\\
&\int e^{i\phi_{\mu\nu\rho}(\xi_1,\xi_2,\xi_3)}\Lambda(\xi_1+\xi_2+\xi_3)^{M+2}m^{\mu\nu\rho}_{k_1,k_2,k_3,l}(\xi_1,\xi_2,\xi_3)d\xi_1d\xi_2d\xi_3,\\
&\Phi_{\mu\nu\rho}(\xi_1,\xi_2,\xi_3)=\Lambda(\xi_1+\xi_2+\xi_3)-\mu\Lambda(\xi_1)-\nu\Lambda(\xi_2)-\rho\Lambda(\xi_3),\\
&\phi_{\mu\nu\rho}(\xi_1,\xi_2,\xi_3)=(x-y)\cdot\xi_1+(x-z)\cdot\xi_2+(x-w)\cdot\xi_3+t\Phi_{\mu\nu\rho}(\xi_1,\xi_2,\xi_3),\\
&m^{\mu\nu\rho}_{k_1,k_2,k_3,l}(\xi_1,\xi_2,\xi_3)=\varphi_l(\xi_2+\xi_3)m_{\mu\nu\rho}(\xi_1,\xi_2,\xi_3)\varphi_{[k_1-1,k_1+1]}(\xi_1)\varphi_{[k_2-1,k_2+1]}(\xi_2)\varphi_{[k_3-1,k_3+1]}(\xi_3)\\
&\text{or }\varphi_l(\xi_1+\xi_2)m_{\mu\nu\rho}(\xi_1,\xi_2,\xi_3)\varphi_{[k_1-1,k_1+1]}(\xi_1)\varphi_{[k_2-1,k_2+1]}(\xi_2)\varphi_{[k_3-1,k_3+1]}(\xi_3).
\end{align*}

From (\ref{m-CL}), $|\nabla^L(\varphi_k\Lambda^M)|\lesssim_L 2^{Mk^+-Lk}$ and (\ref{1/F-CL}) it follows that if $m_{\mu\sigma}$ or $m_{\sigma\rho}\neq|\xi_1+\xi_2+\xi_3|$ then
\[
\left| \nabla^L \left( \Lambda(\xi_1+\xi_2+\xi_3)^{M+2}m^{\mu\nu\rho}_{k_1,k_2,k_3,l}(\xi_1,\xi_2,\xi_3) \right) \right|
\lesssim_L 2^{(M+3)k_3^++2k_3}(1+2^{-L\min(k_i,l)}).
\]
Since $-6j/7<k_i$, $l\le3j/N+O(1)$,
\[
\|\Lambda(\xi_1+\xi_2+\xi_3)^{M+2}m^{\mu\nu\rho}_{k_1,k_2,k_3,l}(\xi_1,\xi_2,\xi_3)\|_{W^{63,1}}\lesssim 2^{(M+3)k_3^++4k_3+2k_2+2k_1+54j}.
\]
Again $|t\nabla\Phi_{\mu\nu\rho}(\xi_1,\xi_2,\xi_3)|<2^{j-3}$,
so by Lemma \ref{int-part-cor} ($L=57$ and $d=6$),
\[
\|1_{\max(|x-y|,|x-z|,|x-w|)>2^{j-2}}\mathcal K(x,y,z,w,t)\|_{L^1_{y,z,w}}
\lesssim 2^{(M+3)k_3^++4k_3+2k_2+2k_1-3j}.
\]
Using $N\ge M+9$ we argue as Case 4.2 of Proposition \ref{W-Z} to get
(note an extra but harmless factor of $j$ coming from summation in $l$)
\[
\left\| \sum_{-6j/7<k_i,l\le3j/N+O(1)}
2^{2j/3}\|Q_j\Lambda^{M+2}B^{\mu\nu\rho\sigma}_{k_1,k_2,k_3,l}(t)\|_{L^2} \right\|_{\ell^2_{j>\mathcal L+5}}
\lesssim (1+t)^{-5/4}\ep_1^3.
\]
If $m_{\mu\sigma}$ or $m_{\sigma\rho}=|\xi_1+\xi_2+\xi_3|$ then
\begin{align*}
&\left| \nabla^L \left( \varphi_k(\xi_1+\xi_2+\xi_3)\Lambda(\xi_1+\xi_2+\xi_3)^{M+2}m^{\mu\nu\rho}_{k_1,k_2,k_3,l}(\xi_1,\xi_2,\xi_3) \right) \right|\\
\lesssim_L&2^{(M+3)k_3^++2k_3}(1+2^{-L\min(k,k_i,l)})
\end{align*}
but the additional assumption $k>-6j/7$ guarantees that we still have
\[
\|1_{\max(|x-y|,|x-z|,|x-w|)>2^{j-2}}P_k^x\mathcal K(x,y,z,w,t)\|_{L^1_{y,z,w}}
\lesssim 2^{(M+3)k_3^++4k_3+2k_2+2k_1-3j},
\]
where $P_k^x$ is the Liitlewood-Paley projection with respect to the $x$ variable. Now the desired bound also follows from summing over $|k|$, $|l|$ and $|k_i|\lesssim j$.

Combining Case 1 through Case 5 above gives
\[
\|H_{\mu\nu\rho}(t)\|_Z
\lesssim (t^{1/5}/R^{4/3}+t^{2+33/N}/R^{10/3-2/N}+(1+t)^{-1.003})\ep_1^3.
\]
Integration in $t$ shows Proposition \ref{H-Z}.
\end{proof}

\end{document}